\documentclass[11pt]{article}
\RequirePackage{etex}
\usepackage{graphicx}
\usepackage{subcaption}
\usepackage{algorithmic}
\usepackage{algorithm}
\usepackage{amsmath}
\usepackage{amssymb}
\usepackage{hyperref}
\usepackage{multirow}
\usepackage{amsfonts}
\usepackage{amssymb}
\usepackage{graphicx}
\usepackage{enumitem}
\usepackage{mathrsfs}
\usepackage{amsthm}
\usepackage{booktabs}
\usepackage{authblk}
\usepackage{cite}

\newcommand{\bn}{\bold{n}}
\newcommand{\bx}{\boldsymbol{x}}

\newcommand{\by}{\boldsymbol{y}}
\newcommand{\br}{\boldsymbol{r}}

\newcommand{\uref}{u^{\text{ref}}}

\newcommand{\bol}{\boldsymbol}

\newcommand{\ney}{\boldsymbol{y}}
\newcommand{\nex}{\boldsymbol{x}}

\newcommand{\nez}{\boldsymbol{z}}

\newcommand{\ner}{\bol{r}}

\newcommand{\de}{\,\mathrm{d}}
\newcommand{\e}{\operatorname{e}}

\newcommand{\inc}{\mathrm{inc}}

\newcommand{\andtext}{\quad\mbox{and}\quad}

\newcommand{\p}{\partial}

\newcommand{\uinc}{u_{\mathrm{inc}}}

\newcommand{\lf}{\left}
\newcommand{\rg}{\right}

\newcommand{\R}{\mathbb{R}}
\newcommand{\C}{\mathbb{C}}

\newcommand{\nor}{\bold n}

\newtheorem{theorem}{Theorem}[section]
\newtheorem{lemma}[theorem]{Lemma}

\newtheorem{remark}[theorem]{Remark}
\newtheorem{definition}[theorem]{Definition}

\usepackage{xcolor}
\usepackage[normalem]{ulem}

\newcommand{\half}{\scalebox{0.9}{$\dfrac{1}{2}$}}

\newcommand{\Dcal}{\mathcal{D}}
\newcommand{\Gcal}{\mathcal{G}}
\newcommand{\Hcal}{\mathcal{H}}

\newcommand{\Lcal}{\mathcal{L}}
\newcommand{\Mcal}{\mathcal{M}}
\newcommand{\Ocal}{\mathcal{O}}
\newcommand{\Pcal}{\mathcal{P}}
\newcommand{\Qcal}{\mathcal{Q}}
\newcommand{\Scal}{\mathcal{S}}
\newcommand{\Wcal}{\mathcal{W}}

\usepackage[width=6.5in, height=8.8in]{geometry}
\usepackage[format=hang,font=small,textfont=sl,labelfont={bf,sl},width=0.95\textwidth]{caption}

\title{General-purpose kernel regularization of boundary integral equations via density interpolation}

 \author[1]{Luiz M. Faria\thanks{luiz.faria@ensta.fr  (corresponding author)}}
 \affil[1]{\small{POEMS (CNRS-ENSTA-INRIA), ENSTA Paris, Palaiseau, France}}
 \author[2]{Carlos P\'erez-Arancibia\thanks{cperez@mat.uc.cl} \thanks{This work was supported by FONDECYT under Grant 11181032.}}
 \affil[2]{\small{Institute for Mathematical and Computational Engineering, Pontificia Universidad Cat\'olica de Chile}}
\author[1]{Marc Bonnet\thanks{mbonnet@ensta.fr}}
\date{\today}

\begin{document}

\maketitle

\begin{abstract}
  This paper presents a general high-order kernel regularization technique applicable to all four integral operators of Calder\'on calculus associated with linear elliptic PDEs in two and three spatial dimensions. Like previous density interpolation methods, the proposed technique relies on interpolating the density function around the kernel singularity in terms of solutions of the underlying homogeneous PDE, so as to recast singular and nearly singular integrals in terms of bounded (or more regular) integrands. We present here a simple interpolation strategy which, unlike previous approaches, does not entail explicit computation of high-order derivatives of the density function along the surface. Furthermore, the proposed approach is kernel- and dimension-independent in the sense that the sought density interpolant is constructed as a linear combination of point-source fields, given by the same {Green's function} used in the integral equation formulation, thus making the procedure applicable, in principle, to any PDE with known {Green's function}. For the sake of definiteness, we focus here on Nystr\"om methods for the (scalar) Laplace and Helmholtz equations and the (vector) elastostatic and time-harmonic elastodynamic equations. The method's accuracy, flexibility, efficiency, and compatibility with fast solvers are demonstrated by means of a variety of large-scale three-dimensional numerical examples.\enlargethispage*{5ex}
\end{abstract}




\section{Introduction}

As is well-known linear elliptic partial differential equations (PDEs) can be
recast in the form of boundary integral equations (BIEs) which can be solved
numerically provided the associated PDE's {free-space or domain-specific Green's
function} is known. While BIE methods offer several advantages over classical
volume discretization techniques, such as finite difference and element methods,
they are also more difficult to understand and implement. The added difficulties
largely stem from the fact that integral operators arising in BIE formulations
give rise to challenging numerical integration problems associated with the
presence of nearly-singular, weakly-singular, singular and hypersingular
kernels. This makes the robust and accurate evaluation of these integral
operators over arbitrary curves or surfaces, a persistent challenge, and the
choice of the ``singularity integration" technique constitutes a fundamental
part of any BIE solution method. Depending on the choice of discretization
method (e.g. Galerkin boundary element method (BEM), Nystr\"om or collocation
BEM) applied to the governing BIE, several techniques are available for dealing
with the singular or nearly singular integrals, such as semi-analytical
methods~\cite{klockner2013quadrature,lenoir:salles:12,beale2001method,Barnett:2014tq,wala2019fast,helsing2008evaluation}, singularity
substraction~\cite{caorsi1993theoretical,graglia1993numerical,jarvenpaa2006singularity,Jarvenpaa:2003cs,wilton1984potential,YlaOijala:2003bn},
Duffy-like
transformations~\cite{duffy1982quadrature,Sauter:2001iw,Sauter:1996iw,reid2015generalized,hackbusch1993efficient},
polar singularity
cancellation~\cite{Bruno:2001ima,Hackbusch:1994tq,schwab1992numerical,bremer2012nystrom,ying2006high},
singularity extraction~\cite{Schulz:1998kl,Schwab:bj}, spectral methods for smooth surfaces represented using spherical harmonics~\cite{ganesh2004high,gimbutas2013fast} and direct
methods~\cite{gui:92,gui:98}. Despite this plethora of options, each of which
targets specific integration problems associated to specific methods, the
``singularity integration" problem remains an active research area due to its
importance and the lack of a universal and robust approach to handle these
integrals. We refer the reader to~\cite{HDI3D,perez2019planewave} for a more
in-depth review of the various existing methods.\enlargethispage*{1ex}

A novel class of singular integral evaluation methods, namely density interpolation methods (DIMs), has been recently put forth in a series of papers addressing Laplace~\cite{HDI3D,gomez2020regularization}, Helmholtz~\cite{plane-wave:2018,perez2019planewave}, and Maxwell equations~\cite{perez2019planewave2}. DIMs are semi-analytical kernel regularization techniques that rely on interpolations of the density function at singular and nearly-singular points on the surface. The density interpolant is devised in the form of a linear combination of solutions of the underlying homogeneous PDE. This density interpolant is then combined with Green's representation formula to recast the integral operators and layer potentials in terms of ``regularized'' operators with integrands whose smoothness at the singular/nearly-singular points is directly controlled by the interpolation order. The aforementioned density interpolant is sought so that it closely mimics a truncated Taylor series on the surface, in the sense that tangential derivatives of the interpolant at the singular/nearly-singular point match
the corresponding derivatives of the density, which are computed numerically in the case of Nystr\"om methods~\cite{plane-wave:2018,HDI3D,perez2019planewave} and analytically in the case of low-order BEM on triangular meshes~\cite{perez2019planewave,perez2019planewave2}. Higher-order Nystr\"om DIMs, in particular, require numerical computation of high-order derivatives of both the density and the parametrization of the curve/surface. Naturally, the accuracy of these tangential derivatives becomes increasingly more sensitive to the quality and smoothness of the surface parametrization as the derivative order increases, thus limiting the accuracy and applicability of these methods when applied to engineering problems involving complex three-dimensional surfaces. DIMs can be viewed as the generalization and systematization of indirect regularization approaches based on integral identities satisfied by simple (constant, affine...) solutions of the underlying zero-frequency PDE, which appeared early in the development stages of the BEM~\cite{RIZZO1967,riz:85,kri:92,B-1992-3} and underlie the exposition in~\cite{bonnet1999boundary}.

This paper presents a new density interpolation approach that tackles in a unified manner the nearly singular, weakly singular, strongly singular, and hypersingular surface integrals arising in BIE formulations of linear elliptic PDEs. The proposed DIM relies on a novel density interpolant whose construction does not entail evaluation of tangential derivatives. Instead, the desired Taylor-interpolation property is approximately achieved by simply matching the density and a linear combination of solutions of the homogeneous PDE at an appropriately selected set of surface nodes at and around the singular/nearly-singular point. By contrast with previously-proposed DIMs, which target specific linear elliptic PDEs, we present here a general methodology where the interpolant is sought as a linear combination of point-source fields given in terms of the {Green's function} involved in the BIE formulation of the PDE. Relying on both higher-order curved triangular meshes~\cite{geuzaine2009gmsh}, non-overlapping quadrilateral patch manifold representation of surfaces~\cite{bruno2018chebyshev,HDI3D}, and local grids corresponding to high-order quadrature rules, we demonstrate through a variety of numerical examples, that the proposed DIM yields a high-order accurate Nystr\"om method for BIEs for the Laplace, Helmholtz, elastostatic, and time-harmonic elastodynamic equations.

The structure of this paper is as follows. Section~\ref{sec:prelim} introduces the notation and presents the four PDEs considered in this paper. Next, Section~\ref{sec:dens-interp} provides a comprehensive description of the proposed density interpolation technique. The details on the construction of the density interpolant functions using point-source fields, together with the algorithmic considerations, are given in Section~\ref{sec:numer-discr}, with supporting error estimates derived in Section~\ref{sec:error_analysis}. Section~\ref{sec:numerics}, finally, presents a variety of numerical examples in two and three spatial dimensions for both scalar and vector problems, some of them involving large-scale models and $\Hcal$-matrix compression.\enlargethispage*{3ex}

\section{Preliminaries}\label{sec:prelim}
We let $\Omega \subset \R^d$, $d=2,3$, be an open and bounded domain with smooth
boundary $\Gamma = \partial \Omega$ which is, for the time being, assumed to be
of class $C^2$. Letting $a$,~$b$ and $f$ be given functions defined on the
boundary $\Gamma$, we let $u: \Omega \to \C^\sigma$ be the solution of the
interior boundary value problem
\begin{subequations}\begin{align}
    \Lcal u(\br)                                  & = 0       & \br \mbox{ in } \Omega,  \\
    a(\bx) (\gamma_0 u)(\bx) + b(\bx)(\gamma_1 u)(\bx) & = f(\bx)  & \bx \mbox{ in } \Gamma,
  \end{align}\label{eq:BVPs}\end{subequations}
where $\Lcal$ stands for either of the following linear differential operators:
\begin{align}
  \label{eq:considered-PDEs}
  \Lcal u =
  \begin{cases}
    -\Delta u                                                                 & \quad\mbox{(Laplace, $\sigma=1$)},       \\
    -\Delta u - (\omega/c)^2u                                                 & \quad\mbox{(Helmholtz, $\sigma=1$)},     \\
    -\mu \Delta u - (\mu + \lambda) \nabla (\nabla \cdot u)                   & \quad\mbox{(elastostatic, $\sigma=d$)},  \\
    -\mu \Delta u - (\mu + \lambda) \nabla (\nabla \cdot u) - \omega^2 \rho u & \quad\mbox{(elastodynamic, $\sigma=d$)},
  \end{cases}
\end{align}
involving the wave velocity $c$ of an acoustic medium, the mass density $\rho$ and Lam\'e's first and second parameters $\lambda,\mu$ of an elastic medium, and a prescribed angular frequency $\omega$. 
\begin{subequations}
The interior Dirichlet and Neumann trace operators $\gamma_0^-$ and $\gamma_1^-$, and their exterior counterparts $\gamma_0^+$ and $\gamma_1^+$, are defined for sufficiently smooth fields $u$ defined in a neighborhood of $\Gamma$ as the limits
\begin{align}
    (\gamma_0^\pm u)(\bx) & := \displaystyle\lim_{\varepsilon \to 0^+} u(\bx \pm \varepsilon \bn(\bx)), \\
    (\gamma_1^\pm u)(\bx) & :=
    \begin{cases}
      \displaystyle\lim_{\varepsilon \to 0^+} \nabla u(\bx \pm \varepsilon \bn(\bx)) \cdot \bn(\bx) & \mbox{for Laplace/Helmholtz},                           \\
      \displaystyle\lim_{\varepsilon \to 0^+}\left\{\lambda \left(\nabla \cdot u(\bx \pm \varepsilon \bn(\bx))\right)\bn(\bx) +2 \mu \nabla u(\bx \pm \right. \\ \left.\varepsilon \bn(\bx)) \cdot \bn(\bx)
      +\mu \bn(\bx) \times \left(\nabla \times u(\bx \pm \varepsilon \bn(\bx))\right)\,\right\}     & \mbox{for elastostatics/elastodynamics},
    \end{cases}
\end{align}
\end{subequations}
(where, as usual,  $\bn(\bx)$ denotes the outward unit normal to $\Omega$ at $\bx\in\Gamma$) and extended by density to suitable Sobolev spaces, see e.g.~\cite{Mclean2000Strongly}. Throughout this article, $\gamma_0$ and $\gamma_1$ will refer to the interior traces if $\Omega$ is bounded, as in problem~\eqref{eq:BVPs}, or to the exterior traces if $\mathbb{R}^d\setminus\Omega$ is bounded, as in most examples of Section~\ref{sec:numerics}.

As is known~\cite{Mclean2000Strongly}, Green-like integral representation formulae can be derived for all boundary value problems of the form~\eqref{eq:BVPs}. Indeed, letting $G$ denote the {free-space Green's function} associated to the operator~$\Lcal$, the following integral representation formula holds:
\begin{equation}
  u(\br) = \Scal[\gamma_1 u](\br) - \Dcal[\gamma_0u](\br) \quad \mbox{for} \quad \br \in \Omega,
  \label{eq:greens-representation}\end{equation}
where the single- and double-layer potentials above are defined for a given density function $\varphi$~as
\begin{subequations}
  \begin{eqnarray}
    \Scal[\varphi](\br) & := &\int_\Gamma {G}(\br, \by)\varphi(\by) \de s(\by),                                 \\
    \Dcal[\varphi](\br) & :=& \int_\Gamma \left(\gamma_{1,\by}{G}(\br, \by)\right)^\top\varphi(\by) \de s(\by),
  \end{eqnarray}
  \label{eq:potentials}
\end{subequations}
respectively. In~\eqref{eq:potentials}, $\gamma_{1,\by}$ denotes the operator $\gamma_1$ applied with respect to the $\by$ variable and acting in the elastic case on each column of the tensor function $G$. Furthermore, taking the traces $\gamma_0u$ and $\gamma_1 u$ in~\eqref{eq:greens-representation} we arrive at
\begin{subequations}
  \label{eq:greens-formulas}
  \begin{align}
    \label{eq:greens-formula-1}
    \half\gamma_0 u(\bx) & = S[\gamma_1 u](\bx) - K[\gamma_0 u](\bx) \andtext                               \\
    \label{eq:greens-formula-2}
    \half\gamma_1 u(\bx) & = K'[\gamma_1 u](\bx) - T[\gamma_0 u](\bx) \quad\mbox{for} \quad \bx \in \Gamma,
  \end{align}
\end{subequations}
where $S,K,K'$ and $T$ are, respectively, the single-layer, double-layer, adjoint double-layer,
and hypersingular operators. Their definition in terms of $G$ (and for a sufficiently regular density $\varphi$) are:
\begin{equation}
  \label{eq:BIOS}
  \begin{split}
    S[\varphi](\bx)  & := \int_\Gamma G(\bx,\by) \varphi(\by) \de s(\by)                                                             \\
    K[\varphi](\bx)  & := {\rm p.v.}\! \int_\Gamma\left(\gamma_{1,\by}G(\bx,\by)\right)^\top
    \varphi(\by) \de s(\by)                                                                                                          \\
    K'[\varphi](\bx) & := {\rm p.v.}\! \int_\Gamma \left(\gamma_{1,\bx}G(\bx,\by)\right)\varphi(\by) \de s(\by)                      \\
    T[\varphi](\bx)  & := {\rm f.p.}\! \int_\Gamma \gamma_{1,\bx}\left(\gamma_{1,\by}G(\bx,\by)\right)^\top \varphi(\by) \de s(\by),
  \end{split}
\end{equation}
where p.v.~and f.p.~indicate Cauchy principal-value and Hadamard finite-part integrals, respectively. The explicit definitions of the {free-space Green's functions} used in~\eqref{eq:potentials} and~\eqref{eq:BIOS} are provided in~Appendix~\ref{app:free_space_Green}.

Table~\ref{tab:singularity-types} presents a summary of the singularities occurring in each kernel of the integral operators defined in~\eqref{eq:BIOS}. We note that in the scalar cases $(\sigma=1)$, the double-layer and adjoint double-layer operators are in fact given (for sufficiently smooth surfaces) in terms of integrable kernels, so that the principal value is not needed in their definition. This summary makes it clear that the integral operators~\eqref{eq:BIOS} have integrands with various degrees of singularity. Quadrature rules designed for smooth integrands cannot be expected to provide a good approximation of such integrals.

\begin{table}
  \centering
  \begin{tabular}{c  c  c  c  c  c }
    \toprule
    $\Lcal$                               & $d$ & $G(\bx, \by )$        & $\gamma_{1,\by}G(\bx,\by)$ & $\gamma_{1,\bx}G(\bx,\by)$ & $\gamma_{1,\bx} \gamma_{1,\by}G(\bx,\by)$ \\
    \midrule
    \multirow{2}{*}{Laplace/Helmholtz}          & 2 & $\Ocal(\log R)$ & $\Ocal(1)$      & $\Ocal(1)$      & $\Ocal(R^{-2})$                     \\

                                                & 3 & $\Ocal(R^{-1})$ & $\Ocal(R^{-1})$ & $\Ocal(R^{-1})$ & $\Ocal(R^{-3})$                     \\
    \midrule
    \multirow{2}{*}{Elastostatic/Elastodynamic} & 2 & $\Ocal(\log R)$ & $\Ocal(R^{-1})$ & $\Ocal(R^{-1})$ & $\Ocal(R^{-2})$                     \\

                                                & 3 & $\Ocal(R^{-1})$ & $\Ocal(R^{-2})$ & $\Ocal(R^{-2})$ & $\Ocal(R^{-3})$                     \\
    \bottomrule
  \end{tabular}
  \caption{Asymptotic behavior as $R = |\bx - \by| \to 0$ ($\bx,\by\in\Gamma$) of the kernels associated with the integral operators defined in~\eqref{eq:BIOS} corresponding to the differential operators introduced in~\eqref{eq:considered-PDEs}.}
  \label{tab:singularity-types}
\end{table}

\begin{remark}
  The developments to follow focus for expository convenience on interior boundary value problems. However, integral equations and representations similar to~\eqref{eq:greens-representation} and~\eqref{eq:greens-formulas} are applicable to exterior problems (i.e., $\Lcal u=0$ for $\br \in \R^d \setminus \overline{\Omega}$) provided a suitable condition is imposed as $|\br| \to \infty$, namely decay conditions (for the zero-frequency cases) or radiation conditions of Sommerfeld~\cite{COLTON:2012} or Kupradze~\cite{kup:79} types (for the time-harmonic cases). The methods proposed in this work therefore apply with minor modifications to exterior problems, and are in fact mostly demonstrated on exterior problems in the examples of~Section~\ref{sec:numerics}.
\end{remark}

\section{Density interpolation}\label{sec:dens-interp}

This section presents a general density interpolation approach for the regularization of the various singular kernels arising in the integral operators defined in~\eqref{eq:BIOS}. The proposed approach enables the integral operators to be recast as ``regularized'' boundary integrals to which elementary quadrature rules are directly applicable and yield high-order convergence.

\subsection{Density interpolant and kernel-regularized boundary integral operators}

We begin by formalizing the definition of a (high-order) density interpolant, and then showing how it can be used to regularize the singular integral operators defined in~\eqref{eq:BIOS}:
\begin{definition}[Density Interpolant]  \label{def:density-interpolant}
  Given a surface  density $\varphi : \Gamma \to \C^\sigma$, the function $\Phi(\br;\bx) : \Omega \times \Gamma \to \C^\sigma$ is said to be a density $(\alpha,\beta)$-interpolant of degree $\varrho\geq 0$
  if  for a given $\nex\in\Gamma$ it satisfies the conditions:
  \begin{equation}\begin{aligned}
      &\text{(a) \ }& \Lcal\Phi(\br;\bx) & = 0 \quad \mbox{for} \quad \br \in \Omega, \\
      &\text{(b) \ }& \gamma_0\Phi(\by;\bx)    & = \alpha \varphi(\bx) + \Ocal(|\bx-\by|^{\varrho+1})\quad \mbox{and} \\
      &\text{(c) \ }& \gamma_1\Phi(\by;\bx)    & = \beta \varphi(\bx) + \Ocal(|\bx-\by|^{\varrho+1})\quad\mbox{as}\quad \Gamma\ni\ney\to\nex\in\Gamma.
    \end{aligned}\label{eq:hom_pde}
    \end{equation}
\end{definition}
For a sufficiently smooth surface $\Gamma$ that locally corresponds to the graph of a smooth function $\chi:\R^{d-1}\to\R^d$ in a neighborhood of $\nex$, it follows by Taylor's theorem that conditions~(\ref{eq:hom_pde}b) and~(\ref{eq:hom_pde}c) are satisfied if
\begin{equation}
  \begin{aligned}
    \partial^\theta\lf( \gamma_0\Phi(\cdot;\bx)\circ \chi \rg) & = \alpha \partial^\theta \lf(\varphi\circ\chi\rg) \andtext \\
    \partial^\theta \lf(\gamma_1\Phi(\cdot;\bx)\circ \chi\rg)  & = \beta \partial^\theta \lf(\varphi\circ\chi\rg)
  \end{aligned} \label{eq:interpolation-condition-taylor}%
\end{equation}
at $\chi^{-1}(\nex) \in\R^{d-1}$, for all $|\theta| \leq \varrho$, where, utilizing the standard multi-index notation, $\partial^\theta $ denotes the derivative $\p^{\theta_1}\p^{\theta_2}\cdots\p^{\theta_{d-1}}$ of order $|\theta| = \theta_1\!+\!\theta_2\!+\!\cdots\!+\!\theta_{d-1}$ with respect to the parameter space variables. We note that in all previous works on DIMs, the construction of the interpolant was achieved by imposing conditions \eqref{eq:interpolation-condition-taylor}, where the required (high-order) surface derivatives were computed in spectrally accurate fashion using a Chebyshev grid.

According to Definition~\ref{def:density-interpolant}, the interpolants $\{\Phi(\cdot;\nex)|\,\nex\in\Gamma\}$ form a parametric family of solutions of the associated homogeneous PDE such that for every $\bx \in \Gamma$ the Dirichlet and Neumann traces of $\Phi(\cdot;\nex)$ provide a local and high-order approximation of $\varphi$ (up to scaling factors $\alpha$ and $\beta$) on the surface. Definition~\ref{def:density-interpolant} formalizes and extends the concept of density interpolant introduced in previous works~\cite{plane-wave:2018,HDI3D,perez2019planewave}. For example, arbitrarily higher-order $(\alpha,\beta)$-interpolants for the Laplace equation are developed in~\cite{HDI3D} where a basis of harmonic polynomials is used to numerically construct them.

The simplest example of a density interpolant is perhaps the lowest-order interpolant $\Phi(\br;\bx) = \varphi(\bx)$ for the Laplace equation, long used by the BEM community. It is not difficult to verify that such a $\Phi$ is in fact a $(1,0)$-interpolant of degree $\varrho=0$, as per Definition~\ref{def:density-interpolant}. A more general class of $(1,1)$-interpolants of order $\varrho=0$ have been put forth by D. Chan and collaborators for the kernel regularization of direct Laplace~\cite{sun2014robust,Klaseboer:2012ks}, Stokes flow~\cite{Klaseboer:2012ks}, elasticity~\cite{klaseboer2019helmholtz} and Helmholtz~\cite{sun2015boundary} boundary integral equation formulations, the latter having then been recently applied to electromagnetic scattering problems~\cite{klaseboer2016nonsingular}. The low-order character of this class of regularization techniques makes them aplicable only to weakly-singular integral operators enabling them to be recast in terms of bounded but not differentiable kernels on which elementary quadrature rules render limited accuracy. Finally, we mention that a $(1,0)$-interpolant of order $\varrho=1$ can be found in~\cite[section 5.3]{bonnet1999boundary} for the regularization of the challenging elastostatic/elastodynamic hypersingular kernels, which is constructed by means of solutions associated to rigid-body rotation and translation. This kernel-regularization approach expresses hypersingular elasticity operators in terms of weakly singular integrals for which a variety of specialized procedures exist for their accurate numerical evaluation,~e.g.~\cite{bruno2018chebyshev,Sauter2010}. 

The following lemma, which forms the basis for density interpolation techniques, is an immediate consequence of the boundary integral identities~\eqref{eq:greens-formulas} applied to the density interpolant and the local properties~(\ref{eq:hom_pde}b) and~~(\ref{eq:hom_pde}c) of the interpolant:
\begin{lemma}[Regularized integral operators]\label{lm:reg}
  Let $\alpha,\beta \in \C$ and consider the following linear combinations of integral operators
  \begin{equation}\label{eq:reg_int_op}
    \begin{aligned}
      V_{\alpha,\beta}[\varphi](\bx) & := \alpha K[\varphi](\bx) - \beta S[\varphi](\bx),  & \\
      W_{\alpha,\beta}[\varphi](\bx) & := \alpha T[\varphi](\bx) - \beta K'[\varphi](\bx), &
    \end{aligned}\end{equation}
  for $\nex\in\Gamma$, where $S$, $K$, $K'$ and $T$ are defined in~\eqref{eq:BIOS},  and let $\Phi$ be a $(\alpha,\beta)$-interpolant of order $p\geq 0$, as per Definition~\ref{def:density-interpolant}. Then, using identities~\eqref{eq:greens-formulas} applied to $\Phi$, they can be recast as
  \begin{equation}\begin{aligned}
      &\text{(a) \ }& V_{\alpha,\beta}[\varphi](\bx) & = K[\alpha \varphi(\cdot) - \gamma_0\Phi(\cdot;\bx)](\bx) - S[\beta \varphi(\cdot) - \gamma_1\Phi(\cdot;\bx)](\bx) - \half\gamma_0\Phi(\bx;\bx),  \\
      &\text{(b) \ }& W_{\alpha,\beta}[\varphi](\bx) & = T[\alpha \varphi(\cdot) - \gamma_0\Phi(\cdot;\bx)](\bx) - K'[\beta \varphi(\cdot) - \gamma_1\Phi(\cdot;\bx)](\bx) - \half\gamma_1\Phi(\bx;\bx).
    \end{aligned}\label{eq:reg-comb-integral-op}\end{equation}
  Furthermore, $\Phi$ being an $(\alpha,\beta)$-interpolant of order $\varrho\geq 0$, we have that
  \begin{equation}\begin{aligned}
      \gamma_0 \Phi(\by;\bx) - \alpha\varphi(\bx) &= \Ocal(|\bx -\by|^{\varrho+1})\andtext\\
      \gamma_1 \Phi(\by;\bx) - \beta\varphi(\bx) &= \Ocal(|\bx - \by|^{\varrho+1}),
    \end{aligned}\end{equation}
as $\nex\to\ney$ and thus the surface  integrands in~\eqref{eq:reg-comb-integral-op} are more regular than those of~\eqref{eq:reg_int_op}. In particular, if
$\gamma_1G(\bx,\by) = \Ocal(|\bx - \by|^{-q_1})$ and $\gamma_0
  G(\bx,\by) = \Ocal(|\bx-\by|^{-q_0})$ as $\by \to \bx$, then the corresponding
integral operators in~(\ref{eq:reg-comb-integral-op}a) are defined in terms of integrands
that behave as $\Ocal(|\bx - \by|^{\varrho-q_1+1})$ and
$\Ocal(|\bx-\by|^{\varrho-q_0+1})$ as $\by \to \bx$. Similarly, if
$\gamma_{1,\bx} G(\bx,\by) = \Ocal(|\bx - \by|^{-q_2})$ and $\gamma_{1,\bx}\gamma_{1,\by}
G(\bx,\by) = \Ocal(|\bx-\by|^{-q_3})$ as $\by \to \bx$, then the corresponding
integral operators in~(\ref{eq:reg-comb-integral-op}b) are defined in terms of integrands
that satisfy $\Ocal(|\bx - \by|^{\varrho-q_2+1})$ and
$\Ocal(|\bx-\by|^{\varrho-q_3+1})$ as $\by \to \bx$.
  \label{lem:regularized-sldl}
\end{lemma}
\begin{remark}
Since the four Calder\'on operators in~\eqref{eq:BIOS} can be expressed in terms of $V_{\alpha,\beta}$ and $W_{\alpha,\beta}$ as $K = V_{1,0}$, $S = V_{0,-1}$, $T = W_{1,0}$ and $K' = W_{0,-1}$, they can be regularized individually by means of the density interpolant. Moreover, linear combinations of the form $V_{\alpha,\beta}$ and $W_{\alpha,\beta}$ occur naturally e.g. in combined field integral equations (CFIEs) for exterior problems, see examples in Section~\ref{sec:numerics}.
\end{remark}
Lemma~\ref{lm:reg} states that, provided a suitable density interpolant is constructed, the boundary integrals~\eqref{eq:BIOS} can be rewritten in terms of integrals with regular (e.g., bounded or differentiable) integrands which can be numerically evaluated by means of elementary quadrature rules. We mention in passing that nearly-singular integrals, where \eqref{eq:greens-representation} is to be evaluated for an observation point which is close to, but not on, the surface $\Gamma$, can be similarly regularized by introducing the orthogonal projection $\bx^{\star} := {\rm arg\,min}_{\ney \in \Gamma}|\ney - \bx^{\star}|$ of $\bx$ on $\Gamma$ and performing the regularization about $\bx^{\star}$. We refer the reader to~\cite{HDI3D} for a more through discussion on the regularization of nearly singular integrals in the context of high-order DIMs.

\subsection{Density interpolation via collocation}\label{sec:dens-interp-colloc}

This section presents the novel high-order approximate density interpolant that, unlike existing Taylor-like density interpolants,  can be constructed by means of a simple collocation procedure which does not entail evaluation of tangential derivatives of the density function.

Specifically, for a given point $\nex\in\Gamma$ we consider a finite set of distinct neighboring points $\{\ney_j\}_{j=1}^P\subset\Gamma_h(\nex)\subset\Gamma$ where $\Gamma_h(\nex)$ is a simply connected subset of $\Gamma$ containing $\nex$ and satisfying ${\rm diam }(\Gamma_h(\nex)) =\sup_{\ney,\nez\in \Gamma_h(\nex)}|\ney-\nez|<h$ for some small $h>0$. As discussed later in Section~\ref{sec:numer-discr}, once a surface discretization is available, $\Gamma_h(\nex)$ is selected as the surface patch to which $\nex$ belongs, $h$ is the characteristic patch diameter, and the neighboring points are chosen as the interior quadrature nodes used for numerical integration over that patch. The interpolant is then sought in the form
\begin{equation}\label{eq:approx_interp_exp}
  \Phi(\ner;\nex) = \sum_{\ell=1}^L G(\ner,\nez_\ell)c_\ell(\nex),
\end{equation}
where $\{\nez_\ell\}_{\ell=1}^{L}\subset \R^d\setminus\overline\Omega$ is a judiciously chosen fixed (independent of $\nex$) set of exterior points placed outside the domain~$\Omega$ and $\{c_\ell(\nex)\}_{\ell=1}^{L}\subset \C^{\sigma}$ is a set of (a priori unknown, possibly vector) coefficients, the total number of scalar coefficients involved in~\eqref{eq:approx_interp_exp} being $\sigma L$ (with $\sigma=1$ for scalar PDEs and $\sigma=d$ for vector PDEs, see~\eqref{eq:considered-PDEs}). This is akin to the method of fundamental solutions~\cite{fairweather1998method,barnett2008stability}.

Since $\Lcal G(\cdot,\nez_\ell)c_\ell(\nex)=0$ in $\Omega$ for each $\ell=1,\ldots,L$, the interpolant $\Phi(\cdot;\nex)$ defined by~\eqref{eq:approx_interp_exp} satisfies $\Lcal \Phi(\cdot;\nex)=0$ in $\Omega$, i.e.~(\ref{eq:hom_pde}a), by construction. To find the coefficients $\{c_\ell(\nex)\}_{\ell=1}^{L}$ we impose the following collocation-interpolation conditions:
\begin{equation}\label{eq:collocation_conds}
  \gamma_0\Phi(\ney_j;\nex) = \alpha\varphi(\ney_j)\quad\mbox{and}
  \quad \gamma_1\Phi(\ney_j;\nex) = \beta\varphi(\ney_j),
  \quad j=1,\ldots,P,
\end{equation}
at all neighboring/collocation points $\{\ney_j\}_{j=1}^P$. The conditions~\eqref{eq:collocation_conds} give rise to a total of $2\sigma P$ linear equations for the $\sigma L$ unknown coefficients, which can be solved by means of the Moore-Penrose pseudo-inverse of the associated matrix provided $L\geq 2P$. More details on the construction of $\Phi$ are given in Section~\ref{sec:interpolant_construct}.

While the regularization of the singular integrals by means of an exact interpolant (see Definition~\ref{def:density-interpolant}) can be easily understood through Lemma~\ref{lem:regularized-sldl}, the analysis of the regularized operators through an approximate interpolant is more subtle. This is because the approximate interpolant no longer enforces conditions of the type $\varphi(\by) - \Phi(\by;\bx) = \Ocal(|\bx - \by|^{p})$ as $\Gamma \ni \by \to \bx \in \Gamma$, and therefore the kernels in the ``regularized'' integral operators are no longer smooth when an approximate interpolant is used. As it turns out, however, the collocation interpolant can effectively regularize weakly-singular, singular, and even hypersingular integral kernels, provided the set of interpolation point~$\{\ney_j\}_{j=1}^P$ is suitably selected. This is so because the resulting interpolation error terms $\alpha \varphi - \gamma_0 \Phi$ and $\beta \varphi - \gamma_1\Phi$, present in the integrands of~\eqref{eq:reg-comb-integral-op}, remain ``small'' within the set (patch) $\Gamma_h(\nex)$ containing  the kernel's singularity.

Therefore, in order to avoid the kernel singularity altogether in the context of the collocation density interpolants, we resort to the following approximations of the regularized integral operators:
\begin{equation}
\begin{aligned}
  V_{\alpha,\beta}[\varphi](\bx)
 \approx \widetilde V_{\alpha,\beta}[\varphi](\bx)
 &:=\int_{\Gamma\setminus\Gamma_h(\nex)}(\gamma_{1,\ney}G(\nex,\ney))^\top\lf\{\alpha\varphi(\ney)-\gamma_0\Phi(\ney;\nex)\rg\}\de s(\ney) \\
 & \quad -\int_{\Gamma\setminus\Gamma_h(\nex)} G(\nex,\ney)\lf\{\beta\varphi(\ney)-\gamma_1\Phi(\ney;\nex)\rg\}\de
      s(\ney)- \half\gamma_0\Phi(\bx;\bx), \\
  W_{\alpha,\beta}[\varphi](\bx)
 \approx \widetilde W_{\alpha,\beta}[\varphi](\bx)
 &:= \int_{\Gamma\setminus\Gamma_h(\nex)} \gamma_{1,\bx}\left(\gamma_{1,\by}G(\bx,\by)\right)^\top\lf\{\alpha\varphi(\ney)-\gamma_0\Phi(\ney;\nex)\rg\}\de s(\ney) \\
 & \quad -\int_{\Gamma\setminus\Gamma_h(\nex)} (\gamma_{1,\nex}G(\nex,\ney))\lf\{\beta\varphi(\ney)-\gamma_1\Phi(\ney;\nex)\rg\}\de s(\ney)- \half\gamma_1\Phi(\bx;\bx).
\end{aligned}\hspace*{-1em}\label{eq:approx_ops}\end{equation}
A proper justification of the approximations~\eqref{eq:approx_ops}, and an assessment of the induced errors, are provided later in Section~\ref{sec:error_analysis}. We first focus on the numerical discretization of the proposed DIMs.

\section{Numerical discretization}\label{sec:numer-discr}

This section presents the details of the proposed numerical methods for the construction of the collocation density interpolant and the evaluation of the regularized boundary integral operators~\eqref{eq:approx_ops}. Since the expressions in~\eqref{eq:approx_ops} bypass the need to evaluate singular integrals of any kind, we develop here a Nystr\"om discretization method based on the direct evaluation of the desingularized integrals~\eqref{eq:approx_ops} by means of elementary quadrature rules, which are used for numerical integration over the surface patches making up~$\Gamma$.

\subsection{Boundary discretization and quadrature}
\label{sec:quadrature}

We start by describing the BIE discretization approach used in this work. The boundary $\Gamma\subset\R^d$, $d=2,3$, is assumed to be given by the union of a finite number of non-overlapping patches $\Gamma_m$, $m=1,\ldots,M$. A canonical example of such a boundary representation would be a triangular surface mesh in three-dimensions formed by $M$ planar triangles. Integrals over the whole boundary $\Gamma$ can thus be expressed as
\begin{equation}
  \int_{\Gamma}F(\ney)\de s= \sum_{m=1}^M\int_{\Gamma_m}F(\ney)\de s,
  \nonumber
\end{equation}
where the function $F:\Gamma\to \C^\sigma$ is assumed to be regular (at least bounded). To approximate such integrals over a given patch $\Gamma_{m}$, we resort to quadrature rules comprising a total of $P_m$ nodes $\Pcal_m = \{\ney_{m,r}\}_{r=1}^{P_m}$, always assumed to lie in the interior of $\Gamma_m$, and weights $\Wcal_m = \{w_{m,r}\}_{r=1}^{P_m}$.

For two-dimensional problems (i.e., one-dimensional line integrals), we resort to standard Gauss-Legendre quadrature rules. In three dimensions, we consider two different quadrature types depending on the shape of the reference patch: (a) tensor products of one-dimensional Gauss-Legendre quadratures for quadrilaterals, and (b) tabulated two-dimensional Gauss quadratures for triangles~\cite{cowper1973gaussian,dunavant1985high}. Figure~\ref{fig:quad_points} depicts the quadrature node locations for patches and rules used in this paper for three-dimensional problems.
\begin{figure}[h] \centering
  \includegraphics[width=0.7\textwidth]{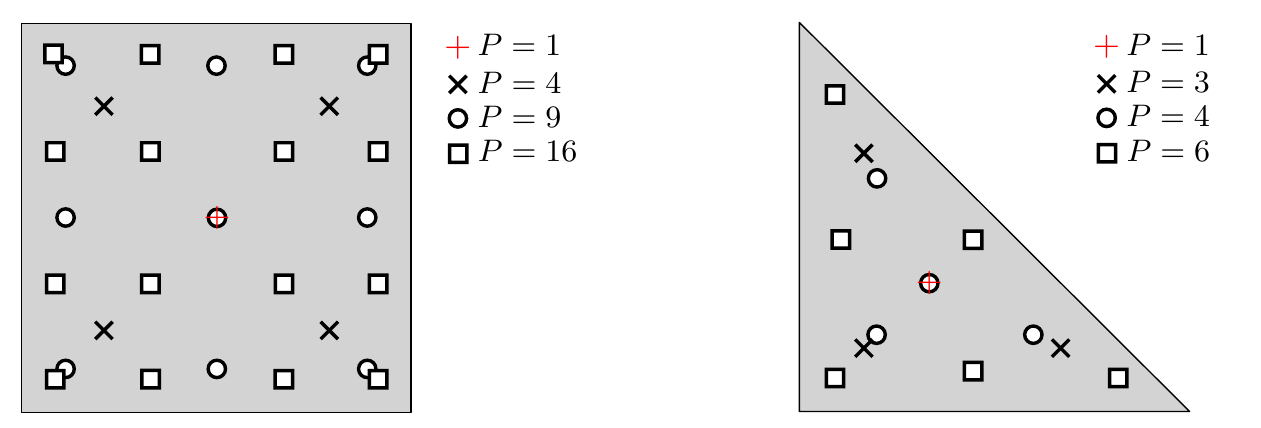}
  \caption{Left: tensor-product Gauss-Legendre quadrature points for integrals over quadrilateral patches. Right: tabulated Gauss quadrature points for integrals over triangular patches~\cite{cowper1973gaussian,dunavant1985high}. $P$ is the total number of quadrature nodes.}\label{fig:quad_points}
\end{figure}

Aggregating the quadrature nodes and weights\footnote{In practice, the normal vector at each quadrature node is also stored since it is needed in the computation of some of the integral kernels involved. For the sake of presentation simplicity, this detail has been omitted.} for each patch, we obtain the global set $\Qcal$ of nodes and weights for the whole boundary $\Gamma$: $\Qcal = \Big\{ \big\{ \ney_{m,r} \big\}_{r=1}^{P_m}, \big\{ w_{m,r} \big\}_{r=1}^{P_m} \Big\}_{m=1}^M$ . In the algorithmic descriptions of the proposed numerical procedures (Sections~\ref{sec:interpolant_construct} and~\ref{sec:forward_map}), it will be convenient to utilize a global indexing system, where the global quadrature is given by $\Qcal = \left\{\by_i, w_i\right\}_{i=1}^N$, with $N = \sum_{m=1}^M P_m$ being the total number of quadrature points. We henceforth denote by $m_i$, where $1 \leq m_i \leq M$, the index of the patch $\Gamma_{m_i}$ to which the node $\by_i$ belongs. Since we use quadrature nodes which lie strictly within the patches, there is a unique patch index $m_i$ for each node $\by_i$, $1 \leq i \leq N$. Unless otherwise stated, we employ the global indexing system of quadrature nodes in the sequel. We additionally define the list $I(m):=\big\{ i : \by_i\in\Gamma_{m} \ \big\}$ of the (global) indices of all nodes lying on a given patch $\Gamma(m)$ and the complementary list $\overline{I}(m):=\big\{\,1,\ldots,N\,\big\}\setminus I(m)$. With this definition, $I(m_i)$ gives the global list of quadrature nodes lying in the same patch as $\by_i$.

We then adopt the Nystr\"{o}m solution method, whereby the relevant integral operators~\eqref{eq:approx_ops} are approximated by applying the above quadrature to the integrals and performing a collocation at the quadrature nodes (i.e. setting $\bx=\by_i$, $1\leq i\leq N$ in~\eqref{eq:approx_ops}), the resulting discretized BIE being solved for the density values at the quadrature nodes.\enlargethispage*{3ex}

Finally, we mention that although the boundary $\Gamma$ has been so far assumed to be globally smooth, in order for the proposed methodology to work (see Section~\ref{sec:error_analysis}) $\Gamma$ is is only required to be piecewise smooth, with each patch $\Gamma_m$ admitting a smooth $C^\infty$ parametrization $\chi_m:\R^{d-1}\to\Gamma_m$.

\subsection{Collocation interpolant construction}\label{sec:interpolant_construct}

This section provides details on the construction of the collocation density interpolant introduced in Section~\ref{sec:dens-interp-colloc}.

\paragraph{Selecting the source points:}
The first task in the construction of the collocation density interpolant~\eqref{eq:approx_interp_exp} is to select an appropriate set of source points $\{\nez_\ell\}_{\ell=1}^L$. The main principle underlying this selection is that a small number $L$ of sources should suffice to produce, via~\eqref{eq:approx_interp_exp}, the density interpolants $\Phi(\ner;\nex)$ associated to each quadrature node $\nex=\ney_i$ on $\Gamma$. In order to achieve that, the set of source points is chosen as the nodes of a high-order quadrature rule on a curve/surface $S$ surrounding $\Gamma$ (Figure~\ref{fig:source-location}). The rationale behind this choice is that the density interpolant can then be interpreted as the approximation
\begin{equation}
  \Phi(\ner;\nex) =\sum_{\ell=1}^L G(\ner,\nez_\ell)c_\ell(\nex)\approx\int_{S}G(\ner,\ney)\tilde c(\ney;\nex)\de s(\ney)
\end{equation}
of a certain integral over $S$. Since the quadrature rule on $S$ exhibits high-order convergence as the number of quadrature nodes increases, a moderate number $L$ of quadrature nodes produces a sufficiently good approximation of $\Phi(\ner;\nex)$ for all $\ner\in\Omega$ and $\nex\in\Gamma$. For the sake of implementation simplicity, $S$ is selected as either a circle ($d=2$) or a sphere ($d=3$).\enlargethispage*{3ex}

To define a suitable surrounding circle/sphere $S$, we consider the minimal (axis-aligned) bounding box $B\subset\R^{d}$ of the set of surface nodes $\left\{\ney_{i}\right\}_{i=1}^N\subset\Gamma$, and denote by $\bx_c$ and $R>0$ its center and radius, respectively (Figure~\ref{fig:source-location}). In two dimensions we then select $L$ source points $\nez_\ell$ on the circle of radius $t R$ and center $\bx_c$, the multiplier $t>1$ ensuring that those points do not lie too close to~$\Gamma$. The points are chosen by sampling $L$ equispaced angles on $[0,2\pi)$ and using polar coordinates to map those onto the circle (Figure~\ref{fig:source-location}); they are therefore associated to the trapezoidal rule on~$S$, which yields superalgebraic convergence for $C^\infty$ integrands. The value $t=5$ is used for all the numerical experiments presented in this paper. We have found the present DIM to be largely insensitive to $t$, with values ranging from $2$ to $1000$ producing approximately the same results. Indeed, in the limit $t\to\infty$, the point source fields used for the Helmholtz and elastodymanic boundary value problems, could be replaced by planewaves in the directions $-\nez_\ell/|\nez_\ell|$, in a manner akin to the planewave DIM~\cite{perez2019planewave}.

Similarly, in three dimensions the $L$ source points $\nez_\ell$ lie on the sphere $S$ with center $\bx_c$ and diameter $t R$ (with $t=5$ in practice); more precisely, the $\{\nez_\ell\}_{\ell=1}^L\subset S$ are taken as Lebedev quadrature nodes~\cite{lebedev1976quadratures}, which are constructed so as to exactly integrate spherical harmonics up to a specified order. We briefly mention here that other selection strategies were also examined, such as points given in spherical coordinates by uniform grids in both azimuthal and polar angles, or uniformly distributed random points on $S$. Overall, although the observed difference between the aforementioned strategies was not significant (e.g., smaller than a factor of 10), the Lebedev points appeared to consistently yield better results, so we settled on that choice.
\begin{figure} \centering
  \includegraphics[width=0.35\textwidth]{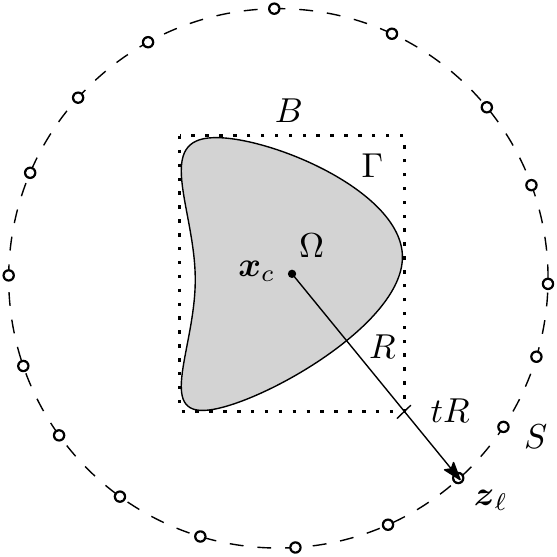}
  \caption{Schematic of source locations used in the interpolant construction.}\label{fig:source-location}
\end{figure}

\paragraph{Computing the coefficients:} With the set $\{z_\ell\}_{\ell=1}^L\subset S$ now selected, we need to find the coefficients $\{c_\ell(\nex)\}_{\ell=1}^L$ in the definition of the density interpolant~\eqref{eq:approx_interp_exp} at~$\nex\in\Gamma$.

Instead of constructing a collocation interpolant $\Phi(\cdot;\nex)$ for each surface node $\nex=\ney_i$, $i=1,\ldots,N$, we construct a single patch interpolant
\begin{equation}
  \Phi_m(\br) = \sum_{\ell=1}^L G(\ner,\nez_\ell)c_{m,\ell}, \quad m=1,\ldots, M.  \label{eq:element-interpolant}
\end{equation}
associated with all the nodes $\Pcal_m=\{\ney_{r,m}\}_{r=1}^{P_m}$ contained in a surface patch $\Gamma_m$,  $m=1,\ldots,M$. The collocation interpolant $\Phi(\ner,\nex)$ introduced in~\eqref{eq:approx_interp_exp} is then taken at the nodes of $\Pcal_m$ as
\begin{equation}
  \Phi(\ner,\nex) := \Phi_m(\ner)\quad\mbox{for}\quad \nex\in\Pcal_m,\  m=1,\ldots,M.
\end{equation}
The $L$ coefficients $\left\{c_{m,\ell}\right\}_{\ell=1}^L$ featured in the patch interpolant $\Phi_m$ in~\eqref{eq:element-interpolant} are obtained by imposing the interpolation-collocation conditions~\eqref{eq:collocation_conds} at all quadrature nodes $\by_j \in \Pcal_m$, i.e. by requiring
$$(\gamma_0\Phi_m)(\br) =\alpha \varphi(\br)\andtext (\gamma_1\Phi_m)(\br) = \beta \varphi(\br)\quad\mbox{for all}\quad \ner\in\Pcal_m=\{\ney_{m,r}\}_{r=1}^{P_m}.$$
For each patch
$\Gamma_m$, $m=1,\ldots,M$, we thus obtain the following  linear system for the unknown set of  coefficients $\{c_{m,\ell}\}_{\ell=1}^L$:
\begin{equation}
  \mathrm{M}_m \mathbf{c}_m =\mathrm{D}_{\alpha,\beta} \mathbf{b}_m,
  \quad \mathbf{c}_m = \big\{ c_{m,1},\ldots,c_{m,L} \big\}^T,
  \quad \mathbf{b}_m = \big\{ \varphi(\ney_{m,1}),\ldots,\varphi(\ney_{m,P_m}) \big\}^T,
  \label{eq:expansion-coefficients-system}
\end{equation}
where
\begin{align}
  \mathrm{M}_m &= \begin{bmatrix} \mathrm{M}^{(0)}_m \\ \mathrm{M}^{(1)}_m \end{bmatrix}, \ \
  [\mathrm{M}^{(0)}_m]_{r,\ell} = \gamma_0G(\ney_{m,r},\nez_\ell),  \ \
  [\mathrm{M}^{(1)}_m]_{r,\ell} = \gamma_{1}G(\ney_{m,r},\nez_\ell),\ \
  1 \leq r \leq P_m,\  1 \leq \ell \leq L, \\
  \mathrm{D}_{\alpha,\beta} &= \begin{bmatrix} \alpha \mathrm{I} \\ \beta \mathrm{I} \end{bmatrix}, \quad
  [\mathrm{I}]_{i,j} = \delta_{ij},\quad 1 \leq i,j \leq P_m.
\end{align}

\begin{remark}\label{rem:matrix_dimensions}
  In the vector case ($\sigma=d$), each ``coefficient'' is a vector $c_{m,\ell}\in \C^{d}$, and the corresponding entries of $\mathrm{M}_m$ are matrices $[\mathrm{M}^{(j)}_m]_{r,\ell}\in \C^{d\times d}$, $j=0,1$. This means that
  $\mathrm{M}_m$, as a matrix with scalar complex entries, has size $(2\sigma P_m) \times (\sigma L)$
  where $\sigma=1$ (scalar problems) or $\sigma=d$ (vector problems), $d=2,3$ being the dimension of the ambient space~$\R^d$. For the sake of presentation simplicity, in what follows we treat $\mathrm{M}_m$ as a matrix of size $(2P_m) \times L$ with entries in $\C^{\sigma\times \sigma}$ and $\mathbf{c}_{m}$ as a vector of length $L$ with entries in $ \C^\sigma$. This allows for the same notation and, in particular, the same indexing system, to be used for all operators introduced in~\eqref{eq:considered-PDEs}.
\end{remark}
With that in mind,  Algorithm~\ref{alg:interpolant-construction} presents the procedure for assembling the local interpolation matrices $\mathrm{M}_m$ and computing the coefficients $\mathbf{c}_m$ for a given patch $\Gamma_m$ using the quadrature nodes~$\Pcal_m$.

\begin{algorithm}[H]
  \caption{Computation of the coefficients of the patch density interpolant~\eqref{eq:element-interpolant}}
  \label{alg:interpolant-construction}
  \begin{algorithmic}[1]
    \REQUIRE Patch quadrature nodes $\Pcal_m = \left\{\ney_{m,r}\right\}_{r=1}^{P_m}$, {free-space Green's function} $G$; source points, $\left\{\nez_\ell\right\}_{\ell=1}^L$, surface density $\varphi$
    \STATE $\mathrm{M}_m \leftarrow $ initialize matrix of size $(2P_m) \times L$
    \STATE $\bold b_m \leftarrow $ initialize vector of size $2P_m$
    \FOR{$r=1$ \TO $P_m$}
    \STATE $[\bold b_m]_r\leftarrow \alpha \varphi(\ney_{m,r})$
    \STATE $[\bold b_m]_{P_m+r} \leftarrow \beta \varphi(\ney_{m,r})$
    \FOR{$\ell=1$ \TO $L$}
    \STATE $[\mathrm{M}_m]_{r,\ell}  \leftarrow (\gamma_0G)(\nez_\ell,\ney_{m,r})$
    \STATE $[\mathrm{M}_m]_{P_m+r,\ell} \leftarrow (\gamma_1G)(\nez_{\ell},\ney_{m,r})$
    \ENDFOR
    \ENDFOR
    \STATE Solve $\mathrm{M}_m\bold c_m = \bold b_m$ \label{alg:line:pinv}
    \RETURN $\bold c_m$
  \end{algorithmic}
\end{algorithm}

The calculation of the interpolant coefficients $c_{m,\ell}$ involves solving the local linear system~\eqref{eq:expansion-coefficients-system}, where the matrix $\mathrm{M}_m$ has size $(2 P_m)\times  L$ over $\C^{\sigma\times\sigma}$. For computational efficiency, one would want to take $L$ as small as possible, while still having $\mathrm{D}_{\alpha,\beta}\mathbf{b}_m$ in the column space of $\mathrm{M}_m$ to guarantee that the system~\eqref{eq:expansion-coefficients-system} is (possibly non-uniquely) solvable, so that the collocation-interpolation conditions~\eqref{eq:collocation_conds} are verified exactly. We have found that the minimum value $L=2P_m$ leads to rank-deficient matrices $\mathrm{M}_m$ for some patches $\Gamma_m$; $L$ is therefore chosen such that $L > 2P_m$, in order to enrich the column space of $\mathrm{M}_m$. In practice, values as small as $L = 2P_m + 1$ appear to be sufficient. In view of the relatively small computational cost of taking a larger $L$, compared to the other parts of the overall method, and of the improved conditioning observed for larger values of $L$, we have set $L \approx 3P$
for all examples presented in this paper, where $P = \max_{1 \leq m \leq M} P_m$. While we found practical use for allowing the number of quadrature points per patch $P_m$ to be variable (see the examples in Section~\ref{sec:hybrid-meshes} where hybrid meshes are considered), we saw no clear advantage of letting $L$ vary from patch to patch; consequently, $L$ is taken to be constant for all patches.

The solution of the local system~\eqref{eq:expansion-coefficients-system} for each patch $\Gamma_m$, can be obtained by means of the $LQ$ decomposition of the matrix $\mathrm{M}_m$: letting $(p,q):=(2\sigma P_m,\sigma L)$ denote the size of $\mathrm{M}_{m}$ , we have
\begin{equation}
   \mathrm{M}_m = \begin{bmatrix} \mathrm{L} & \mathbf{0} \end{bmatrix} \mathrm{Q}
 = \begin{bmatrix} \mathrm{L} & \mathbf{0} \end{bmatrix}
   \begin{bmatrix} \mathrm{Q}_1 \\ \mathrm{Q}_2 \end{bmatrix}
 = \mathrm{L} \mathrm{Q}_1,
\end{equation}
where $\mathrm{L}\in\C^{p\times p}$ is a lower triangular matrix and $\mathrm{Q}\in\C^{q\times q}$ is a unitary matrix.
The matrix $\mathrm{Q}_1\in\C^{p\times q}$ consists of the first $p$ rows of $\mathrm{Q}$. The coefficients $\mathbf{c}_m$ are then found by the steps
\begin{enumerate}[topsep=3pt,itemsep=3pt,parsep=0pt,partopsep=0pt]
  \item Solving $\mathrm{L}\mathbf{y} = \mathbf{b}_m$ using forward
        substitution, where $\mathbf{y}$ is a temporary vector of size $m \times 1$,
  \item Performing the matrix-vector product  $\mathbf{c}_m=\mathrm{Q}_1^{\star}\mathbf{y}$,
\end{enumerate}
which amount to setting $\mathbf{c}_m=\mathrm{M}^{\dagger}_m\mathbf{b}_m$ with the pseudo-inverse $\mathrm{M}^{\dagger}_m$ of $\mathrm{M}_m$ given by $=Q_1^{\star}L^{-1}$. Note that in the case $\sigma=d$, this step requires $\mathrm{M}_m$ to be converted to a standard matrix over $\C$. Note also that the $LQ$ decomposition bypasses the need to explicitly invert $\mathrm{L}_m$, which due to the possibly bad conditioning of the matrix $\mathrm{M}_m$, could also be poorly conditioned.

The dominant cost of solving the local system~\eqref{eq:expansion-coefficients-system} is then given by the $LQ$ decomposition, which has complexity $\Ocal(LP_m^2)$ using the Householder algorithm. Under the assumption that $L = \Ocal(P_m)$, this gives a dominant cost of $\Ocal(P_m^3)$ per surface patch $\Gamma_m$. Since there are $M = \Ocal(N/P)$ patches, the cost of computing the coefficients $c_{m,\ell}$ for all patches is of the order $\Ocal(N P^2)$, and therefore linear in the number of quadrature nodes (and problem size) $N$.

\subsection{Operator evaluation}\label{sec:forward_map}

With the collocation density interpolant constructed, we are now in a position to evaluate the regularized integral operators $\widetilde V_{\alpha,\beta}$ and $\widetilde W_{\alpha,\beta}$ defined in~\eqref{eq:approx_ops}. We focus on the task of evaluating the operator $\widetilde V_{\alpha,\beta}$, which is a linear combination of the single- and double-layer operators, applied to a given density $\varphi$. The evaluation of $\widetilde W_{\alpha,\beta}\varphi$ follows a completely analogous process.\enlargethispage*{5ex}

\paragraph{Direct operator evaluation:} Fixing a target collocation point $\nex=\by_i$, let $\Gamma_{m_i}$ be the patch containing $\by_i$, and let $\Phi(\cdot;\nex)=\Phi_{m_i}$ denote the collocation density interpolant~\eqref{eq:element-interpolant} associated with the quadrature nodes of $\Gamma_{m_i}$. Since solving the local linear system~\eqref{eq:expansion-coefficients-system} allows to satisfy all interpolation-collocation conditions~\eqref{eq:collocation_conds}, the operator evaluation~$V_{\alpha,\beta}[\varphi](\by_i)$ is approximated by $\widetilde V_{\alpha,\beta}[\varphi](\by_i)$ as given by~\eqref{eq:approx_ops} with $\Gamma_h(\bx)=\Gamma_{m_i}$, i.e.:
\begin{multline}
  \widetilde V_{\alpha,\beta}[\varphi](\by_i)
 = -\half\gamma_0\Phi_{m_i}(\by_i) +\sum_{j\in\overline{I}(m_i)}  w_{j}\big(\gamma_{1,\by}G(\by_i,\by_{j})^T\left\{\alpha \varphi(\by_{j}) - \gamma_0\Phi_{m_i}(\by_{j}) \right\} \\
  - G(\by_i,\by_{j})\left\{\beta \varphi(\by_{j}) - \gamma_1 \Phi_{m_i}(\by_{j}) \right\} \big),\quad i=1,\ldots,N.
    \label{eq:reg-forward-map}
\end{multline}
The validity of ignoring the ``self-patch'' contribution will be supported by theoretical arguments (Section~\ref{sec:error_analysis}) and computational evidence (Section~\ref{sec:numerics}). A straightforward (naive) implementation of formula~\eqref{eq:reg-forward-map} is shown in Algorithm~\ref{alg:forward-map}, where for each target point $\by_i \in \Qcal$ we simply replace the integrand by its kernel-regularized version, and then integrate over the surface (lines~\ref{lst:line:quad-begin}--\ref{lst:line:quad-end}).
\begin{algorithm}
  \caption{Integral operator evaluation}
  \label{alg:forward-map}
  \begin{algorithmic}[1]
    \REQUIRE Quadrature $\Qcal = \left\{\ney_{i},w_i \right\}_{i=1}^{N}$; density $\varphi$\;
    \STATE $\bold b \leftarrow$ initialize zero vector of length $N$
    \STATE $\left\{\Phi_{m}\right\}_{m=1}^M \leftarrow$ construct patch interpolant for each patch $\Gamma_{m}$ (see Algorithm~\ref{alg:interpolant-construction})
    \FOR{$i=1$ \TO $N$}
    \STATE $m_i \leftarrow$ patch index of $\by_i$
    \FOR{$j\in\overline{I}(m_i)$} \label{lst:line:quad-begin}
    \STATE $[\bold b]_i \leftarrow [\bold  b]_i + w_{j}\left(\gamma_{1,\by}G(\by_i,\by_j)\left\{\alpha \varphi(\by_j) - \gamma_0\Phi_{m_i}(\by_j) \right\} - G(\by_i,\by_j)\left\{\beta \varphi(\by_j) - \gamma_1 \Phi_{m_i}(\by_j) \right\} \right)$
    \ENDFOR \label{lst:line:quad-end}
    \STATE $[\bold b]_i \leftarrow [\bold b]_i - \Phi_{m_i}(\by_{i})/2$
    \ENDFOR
    \RETURN $\bold b$
  \end{algorithmic}
\end{algorithm}

\paragraph{Efficient operator evaluation.}

Algorithm~\ref{alg:forward-map} has a $O(N^2)$ complexity, incompatible with fast BIE solution methods, and a more efficient version is therefore called for. We now reorganize it in a way that will facilitate the use of fast operator evaluation methods, and also allow to precompute some of the operations associated with the collocation density interpolant. The latter feature is beneficial when repeated operator evaluations (for different input density functions) are needed, as when iterative linear algebra solvers (such as GMRES) are applied to the discrete BIE system.

These enhancements are based on the following splitting of~\eqref{eq:reg-forward-map}: letting $\bol{\varphi} \approx \big\{\varphi(\by_1),\ldots,\varphi(\by_N)\big\}^T$, we set
\begin{equation}
  \widetilde  V_{\alpha,\beta}[\varphi](\by_{i})
 = \lf[\mathrm{V}_{\alpha,\beta}\bol{\varphi}\rg]_i
 := \big[\mathrm{V}_{\alpha,\beta}^{(0)}\bol{\varphi}\big]_i
 +  \big[\mathrm{V}_{\alpha,\beta}^{(1)}\bol{\varphi}\big]_i,\quad i=1,\ldots,N,
  \label{eq:splitting}
\end{equation}
where the matrix $\mathrm{V}_{\alpha,\beta}^{(0)}$ is given by
\begin{align}
  \mathrm{V}_{\alpha,\beta}^{(0)} & = \left(\alpha  \mathrm{K}^{(0)} - \alpha \mathrm{S}^{(0)}\right)\text{diag}(\bold w),
\end{align}
with
\begin{equation}
\begin{aligned}
  \big[\mathrm{K}^{(0)}\big]_{i,j} = &
  \begin{cases}
    \gamma_{1,\by}G(\by_i,\by_j)\quad & \mbox{for} \quad j\in\overline{I}(m_i), \\
    0 \quad                           & \mbox{for} \quad j\in I(m_i),
  \end{cases}                                                                                               \\
  \big[\mathrm{S}^{(0)}\big]_{i,j} = & \begin{cases}
    G(\by_i,\by_j)\quad & \mbox{for} \quad j\in\overline{I}(m_i), \\
    0 \quad             & \mbox{for} \quad j\in I(m_i),
  \end{cases}
\end{aligned} \quad \andtext [\bold w]_{j} = w_j.
\label{eq:S0_K0_definition}
\end{equation}

We can view $\mathrm{V}_{\alpha,\beta}^{(0)}$ as an approximation of the complete matrix $\mathrm{V}_{\alpha,\beta}$ that ignores the ``self-patch'' contributions: the matrix $\mathrm{V}_{\alpha,\beta}^{(0)}$ has a zero block diagonal and is otherwise dense. Since this approximation retains all far-field contributions to $\mathrm{V}_{\alpha,\beta}$, $\Hcal$-matrix or fast-multipole methods can be applied to accelerate evaluations of $\big[\mathrm{V}_{\alpha,\beta}^{(0)}\bol{\varphi}\big]_i$.\enlargethispage*{7ex}

Then, the matrix $\mathrm{V}_{\alpha,\beta}^{(1)}$ can be understood as a correction to $\mathrm{V}_{\alpha,\beta}^{(0)}$, which accounts for the local kernel-regularization performed around the diagonal (singular) entries by gathering all the contributions arising from the collocation density interpolants. It will be found to be sparse (block-diagonal). Unlike $\mathrm{V}_{\alpha,\beta}^{(0)}$, the explicit form of the entries of $\mathrm{V}_{\alpha,\beta}^{(1)}$ cannot be straightforwardly given and the actual procedure for its evaluation is postponed to Section~\ref{sec:assembl-integr-oper}. Nevertheless, it follows from~\eqref{eq:reg-forward-map} and~\eqref{eq:splitting} that\enlargethispage*{1ex}
\begin{equation}
  \big[\mathrm{V}_{\alpha,\beta}^{(1)}\bol{\varphi}\big]_i
 = - \half\Phi_{m_i}(\by_{i}) + \sum_{j\in\overline{I}(m_i)}  w_{j}\left(\gamma_{0}G(\by_{i},\by_{j})\gamma_1\Phi_{m_i}(\by_{j}) - \gamma_{1,\by}G(\by_{i},\by_{j})\gamma_{0}\Phi_{m_i}(\by_{j})  \right), \label{eq:correction-operator}
\end{equation}
where the dependency of~\eqref{eq:correction-operator} on the density vector
$\bol{\varphi}$ is implicit through the density interpolant~$\Phi_{m_i}$. The specific form~\eqref{eq:element-interpolant} of $\Phi_{m_i}$ suggests the definition of the following matrix $\Theta$, of size
$N\times L$:
\begin{equation}
  \label{eq:theta-ell}
  [\Theta]_{i,\ell} = -\half G(\nez_\ell,\by_{i})
  + \sum_{j\in\overline{I}(m_i)} w_{j}\left(\gamma_{0}G(\by_{i},\by_{j})\gamma_1G(\nez_\ell,\by_{j})
  - \gamma_{1,\by}G(\by_{i},\by_{j})\gamma_0G(\nez_\ell,\by_{j})  \right).
\end{equation}
With this definition, the evaluation of $\big[\mathrm{V}_{\alpha,\beta}^{(1)}\bol{\varphi}\big]_i$ becomes
\begin{equation}
  \big[\mathrm{V}_{\alpha,\beta}^{(1)}\bol{\varphi}\big]_i
  = \sum_{\ell=1}^L [\Theta]_{i,\ell} c_{m_i,\ell}, \label{eq:correction-forward-map-fast}
\end{equation}
where only the set of coefficients $\{c_{m_i,\ell}\}$, which depends on the density $\varphi$, needs to be updated when $\varphi$ change. This reformulation suggests an ``offline'' stage for the evaluation of the operator $ V_{\alpha,\beta}$, where the matrix $\Theta$ as well as the factorized forms $\mathrm{M}_m=\mathrm{LQ}_1$ of the matrices $\mathrm{M}_m$ involved in the local (patch-based) systems~\eqref{eq:expansion-coefficients-system} governing the coefficients $c_{m_i,\ell}$, are precomputed. This ``offline'' stage is summarized in Algorithm~\ref{alg:forward-map-precompute}, where the availability of a routine performing (fast) evaluations of $\lf[\mathrm{V}^{(0)}_{\alpha,\beta}\bol{\varphi}\rg]$ on given (discrete) densities $\bol{\varphi}$ is assumed. The latter requirement may involve some additional offline pre-computation work, in particular if $\Hcal$-matrix compression is used.\enlargethispage*{3ex}

Once the required quantities are computed, the correction $\big[\mathrm{V}_{\alpha,\beta}^{(1)}\bol{\varphi}\big]_i$ can be inexpensively calculated (i.e., in order $\Ocal(N)$ complexity) through~\eqref{eq:correction-forward-map-fast}. The cost of the precomputation, on the other hand, depends on the availability of an acceleration routine since evaluations of $\mathrm{V}_{1,0}$ and $\mathrm{V}_{0,-1}$ on $L$ different densities are needed in order to compute $\Theta$ (see line~\ref{alg:line:forward-map} in Algorithm~\ref{alg:forward-map-precompute}). The overall cost of the pre-computation is then of the same order as that of an evaluation of $V^{(0)}_{\alpha,\beta}$.
\begin{algorithm}[b]
  \caption{Pre-computations for operator evaluations}
  \label{alg:forward-map-precompute}
  \begin{algorithmic}[1]
    \REQUIRE Quadrature $\Qcal$, operator $V_{\alpha,\beta}$, routine for (fast) evaluation of $\bol{\varphi}\mapsto\lf[\mathrm{V}^{(0)}_{\alpha,\beta}\bol{\varphi}\rg]$
    \STATE $\left\{ \nez_\ell \right\}_{\ell=1}^L \leftarrow$ compute source locations
    \STATE $\mathrm{B} \leftarrow$ initialize an $N \times L$ matrix
    \STATE $\mathrm{C} \leftarrow$ initialize an $N \times L$ matrix
    \FOR{$\ell=1$ \TO $L$}
    \FOR{$i=1$ \TO $N$}
    \STATE $\big[\mathrm{B}\big]_{i,\ell} \leftarrow G(\nez_\ell,\by_i)$
    \STATE $\big[\mathrm{C}\big]_{i,\ell} \leftarrow \gamma_1G(\nez_\ell,\by_i)$
    \ENDFOR
    \ENDFOR
    \STATE $\Theta \leftarrow \frac{-\mathrm{B}}{2} + \mathrm{V}^{{(0)}}_{1,0}\mathrm{B} - \mathrm{V}^{{(0)}}_{0,-1}\mathrm{C}$ \qquad (columnwise evaluations of the form $\bol{\varphi}\mapsto\lf[\mathrm{V}^{(0)}_{\alpha,\beta}\bol{\varphi}\rg]$) \label{alg:line:forward-map}
    \FOR{$m=1$ \TO $M$}
    \STATE $\mathrm{M} \leftarrow$ assemble interpolation matrix, see algorithm~\ref{alg:interpolant-construction}
    \STATE $\mathrm{L}_m,\mathrm{Q}_m \leftarrow$ precompute LQ decomposition of $M$
    \ENDFOR
    \RETURN $\mathrm{S}^{(0)}, \mathrm{K}^{(0)}, \Theta, \{\mathrm{L}_p,\mathrm{Q}_p\}_{p=1}^M$
  \end{algorithmic}
\end{algorithm}

\subsection{Assembling the integral operator}\label{sec:assembl-integr-oper}

We jsut saw how the density interpolation procedure recasts the operator $\mathrm{V}_{\alpha,\beta}$ as the sum of a dense matrix $\mathrm{V}_{\alpha,\beta}^{(0)}$, and a correction $\mathrm{V}_{\alpha,\beta}^{(1)}$, and only considered evaluations for the latter. We now describe how $\mathrm{V}_{\alpha,\beta}^{(1)}$ may be assembled as a sparse matrix, assuming $\mathrm{S}^{(0)},\mathrm{K}^{(0)},\Theta, \{\mathrm{L}_m, \mathrm{Q}_m\}_{m=1}^M$ to have been pre-computed as per Algorithm~\ref{alg:forward-map-precompute}.

Consider the $i$-th row of $\mathrm{V}_{\alpha,\beta}^{(1)}$. Since $\Phi_{m_i}$ is computed using only values of $\bol{\varphi}_j$ for $j \in I(m_i)$, it is easy to see that $\big[ \mathrm{V}_{\alpha,\beta}^{(1)} \big]_{i,j}=0$ if $j \not \in I(m_i)$, i.e. the matrix $\mathrm{V}_{\alpha,\beta}^{(1)}$ is block-diagonal. Under the assumption that $P_{m_i} = \# I(m_i) \ll N$, $\mathrm{V}_{\alpha,\beta}^{(1)}$ is thus sparse, as expected. To compute the nonzero entries $\big[ \mathrm{V}_{\alpha,\beta}^{(1)} \big]_{i,j}$ of the $i$-th row, we define the row vector $\bol{\Theta}_i = [\Theta_{i,1},\ldots, \Theta_{i,L}]$, and observe that
\begin{equation}
  \big[\mathrm{V}_{\alpha,\beta}^{(1)}\bol{\varphi}\big]_i
 = \sum_{j\in I(m_i)} \big[\mathrm{V}_{\alpha,\beta}^{(1)}\big]_{i,j}\bol{\varphi}_j
 = \sum_{\ell=1}^L \Theta_{i,\ell}c_{m_i,\ell} = \bol{\Theta}_i \mathbf{c}_{m_i}.
\end{equation}
Using~\eqref{eq:expansion-coefficients-system} to express $\mathbf{c}_{m_i}$ in terms of the density $\bol{\varphi}$ yields
\begin{equation}
  \sum_{j\in I(m_i)} \big[\mathrm{V}_{\alpha,\beta}^{(1)}\big]_{i,j}\bol{\varphi}_j
 = \bol{\Theta}_i \mathrm{M}_{m_i}^\dagger \mathrm{D}_{\alpha,\beta} \big\{ \bol{\varphi} \big\}_{m_i}
\end{equation}
where $\big\{ \bol{\varphi} \big\}_{m_i}:=\big\{ \bol{\varphi}_j \big\}^T_{j\in I(m_i)}$.
This gives that the non-zero entries of the $i$-th row are $\big[\mathrm{V}_{\alpha,\beta}^{(1)}\big]_{i,j} = [\bol{\Theta}_{i}\Mcal_{m_i}^\dagger \mathrm{D}_{\alpha,\beta}]_n$ for $1 \leq n \leq P_{m_i}$, $j$ being the global index of the $n$-th quadrature node of $\Gamma_{m_i}$.

It is important to mention that the product $\mathbf{w} = \bol{\Theta}_{i}\mathrm{M}_{m_i}^\dagger \mathrm{D}_{\alpha,\beta}$, which can be interpreted as the regularized quadrature weights, is computed through the $LQ$ decomposition of $\mathrm{M}_{m_i}$ as $\mathbf{w} = \left(\left(\bol{\Theta}_{i}Q^{\star}\right)L^{-1}\right)D_{\alpha,\beta}$, where the parentheses specify the order in which the operations are performed. The explicit inversion of $\mathrm{L}$ is avoided due to the bad conditioning of $\mathrm{M}_{m_i}$, multiplications by $\mathrm{L}^{-1}$ being performed via forward substitutions. As shown in the results of Section~\ref{sec:numerics}, we did not observe any numerical stability issues in our calculations when using the described procedure, and were able to obtain convergence to tolerances smaller than $10^{-10}$ in many of the examples presented without resorting to e.g. higher-precision arithmetic.


\section{Error estimates}\label{sec:error_analysis}

This section aims at justifying and assessing the errors in the approximations introduced in~\eqref{eq:approx_ops} when
the collocation interpolant is used, and implemented by the algorithms developed in Section~\ref{sec:numer-discr}.

In order to examine these errors, it is helpful to consider the following one-dimensional weakly-singular, singular (principal value), and hypersingular (finite part) integral operators:
\begin{subequations}
  \label{eq:1d_integrals}
  \begin{align}
    I_0[g](t) = & ~\int_{a}^{b}g(s)\log|t-s|\de s,                                \\
    I_1[g](t)=  & ~{\rm p.v.}\!\int_{a}^{b}\frac{g(s)}{t-s}\de s,\label{eq:id_PV} \\
    I_2[g](t) = & ~{\rm f.p.}\!\int_{a}^{b}\frac{g(s)}{(t-s)^2}\de s,
  \end{align}
\end{subequations}
where $a < t < b$ and $b-a = \delta  < 1$. Regarding these operators we have the next lemma:
\begin{lemma}\label{lem:1D_IO}
  Let $g\in C^{1,1}[a,b]$. Then the estimates
  \begin{subequations}\begin{align}
      |I_0[g](t)|\leq & ~2  \delta\{1+|\log (\mu\delta)|\}\|g\|_{C^{0}[a,b]},                                             \\
      |I_1[g](t)|\leq & ~ \delta|g|_{C^{0,1}[a,b]}+|\log\mu|\|g\|_{C^0[a,b]},                                    \\
      |I_2[g](t)|\leq & ~\frac{2}{\mu\delta}\|g\|_{C^0[a,b]}+\delta|g|_{C^{1,1}[a,b]}+|\log\mu|\|g\|_{C^1[a,b]},
    \end{align}\end{subequations}
  hold true for  all $t\in [a+\mu\delta,b-\mu\delta]$, where $0<\mu<\frac12$.
\end{lemma}
\begin{proof} For the first operator $I_{0}$, the integrability of the logarithmic kernel yields
  \begin{equation}
    |I_0[g](t)|\leq \int_{a}^{b}|g(s)\log|t-s||\de s
    \leq \|g\|_{C^{0}[a,b]}\int_{a}^b\!\!|\log|t-s||\de s \leq 2  \delta\{1+|\log (\mu\delta)|\}\|g\|_{C^{0}[a,b]},
  \end{equation}
  where we used  that $\delta<1$. For the integral $I_{1}[g](t)$ we have that the Cauchy principal-value integral can be expressed as
  \begin{align}
    I_1[g](t) = \int_{a}^{b} \frac{g(s)-g(t)}{t-s} \de s + g(t)~{\rm p.v.}\!\int_{a}^{b}\frac{1}{t-s} \de s
    = \int_{a}^{b} \frac{g(s)-g(t)}{t-s} \de s + g(t)\log\left(\frac{t-a}{b-t}\right),
  \end{align}
  which can be bounded as $$|I_1[g](t)| \leq \delta|g|_{C^{0,1}[a,b]}+|\log\mu|\|g\|_{C^0[a,b]}.$$
  Finally, for the finite-part integral operator we have that it can be expressed as
  $$
    I_2[g](t) = -\frac{g(a)}{t-a}-\frac{g(b)}{b-t}-{\rm p.v.}\int_{a}^{b} \frac{g^{\prime}(s)}{t-s} \mathrm{d} s,
  $$
  and hence, using the bound for the principal value integral, we obtain
  \begin{equation}\label{eq:estimate3}
    |I_2[g](t)| \leq \frac{2}{\mu\delta}\|g\|_{C^0[a,b]}+\delta|g|_{C^{1,1}[a,b]}+|\log\mu|\|g\|_{C^1[a,b]}.
  \end{equation}
  The proof is now complete.
\end{proof}

As it turns out, Lemma~\ref{lem:1D_IO} can be used to estimate the neglected quantities in the approximations~\eqref{eq:approx_ops} of each one of the four integral operators and PDEs considered, in the $d=2$ case. Indeed, in order to estimate such  errors we utilize the following result:\enlargethispage*{5ex}
\begin{lemma}\label{lem:d2case}
  Let $\Phi(\cdot;\nex)$ be the collocation density interpolant of a smooth density function $\varphi$ at $\nex\in \Gamma_h(\nex)\subset\Gamma\subset\R^2$, constructed from a set of distinct collocation points $\{\ney_j\}_{j=1}^P\subset \Gamma_h(\nex)$. Let also $\chi:[a,b]\to \Gamma_h(\nex)$ be a smooth bijective parametrization of $\Gamma_h(\nex)$ and let $\nex=\chi(t)$, $t\in(a,b)$.  Then, the interpolation error functions
  \begin{equation}\label{eq:functions_to_estimate}
    e_\alpha(s) := \alpha\varphi(\chi(s)) - \gamma_0\Phi(\chi(s);\nex)\andtext e_\beta(s):=\beta\varphi(\chi(s)) - \gamma_1\Phi(\chi(s);\nex),
  \end{equation}
  satisfy
  \begin{align}
    |e^{(k)}_\alpha(s)|\lesssim h^{P-k}\andtext |e^{(k)}_\beta(s)|\lesssim h^{P-k},
  \end{align}
  for all $s\in [a,b]$ and $k=0,\ldots, P$,  provided   $|F^{(k)}_\ell(s)|\leq c_\ell$, with $F_\ell(s) = \gamma_\ell\Phi(\chi(s);\nex)$, for some $h$-independent constants $c_\ell$, $\ell=0,1$.
\end{lemma}
\begin{proof}
Let $L(s):[a,b]\to\C$ be the $(P-1)$th-degree Lagrange interpolation polynomial of the function  $f(s) = \alpha\varphi(\chi(s))$ in the parameter space at the (distinct) points $\{s_j\}_{j=1}^P\subset (a,b)$ given by $s_j = \chi^{-1}(\nex_j)$, $j=1,\ldots, P$. By the interpolation-collocation conditions~\eqref{eq:collocation_conds} we have that  $L(s)$ is also the interpolation polynomial of $F_0(s) = \Phi(\chi(s);\nex)$. Therefore, using a well-known result of Lagrange interpolation~\cite[Sec.~5, Theorem~1]{isaacson:1994}, we  obtain
\begin{multline}
      \big|e^{(k)}_\alpha(s) \big| = \big| f^{(k)}(s) - F^{(k)}_0(s) \big|
 \leq |f^{(k)}(s)-L^{(k)}(s)|+|F^{(k)}_0(s)-L^{(k)}(s)| \\
 \leq \frac{(b-a)^{P-k}}{(P-k)!}(c_f + c_0),\quad k=0,\ldots, P-1, \label{eq:ineq_d2}
\end{multline}
where $|f^{(P)}(s)|\leq c_f$ and $|F^{(P)}_0(s)|\leq c_0$ for $s\in [a,b]$. Now, since $(b-a)\leq c_\chi h$, where $c_\chi$ is the Lipschitz continuity constant of $\chi^{-1}$, we conclude that  $|e^{(k)}_\alpha(s)| \lesssim h^{P-k}$.  The terms $e_\beta^{(k)}(s)$ can be estimated in a similar manner.
\end{proof}

We are now in a position to present the main result of this section:
\begin{theorem}\label{tm:d2case} Let $\Phi(\cdot;\nex)$ as in Lemma~\ref{lem:d2case} and assume that ${\rm dist}(\nex,\p \Gamma_h(\nex))>\mu h$ for some $h$-independent constant~$0<\mu<1$. Then, the error estimates
  \begin{subequations}\begin{align}
      \big|V_{\alpha,\beta}[\varphi](\nex)-\widetilde V_{\alpha,\beta}[\varphi](\nex)\big|\lesssim & ~h^{P+1}|\log h|\quad\mbox{and}\label{eq:SL_ERROR} \\
      \big|W_{\alpha,\beta}[\varphi](\nex)-\widetilde W_{\alpha,\beta}[\varphi](\nex)\big|\lesssim & ~h^{P-1}\label{eq:HS_ERROR}
    \end{align}
    hold for the Laplace and Helmholtz integral operators, and
    \begin{align}
      \big|V_{\alpha,\beta}[\varphi](\nex)-\widetilde V_{\alpha,\beta}[\varphi](\nex)\big|\lesssim & ~ h^{P}\quad\mbox{and}\label{eq:SL_ERROR_VEC}  \\
      \big|W_{\alpha,\beta}[\varphi](\nex)-\widetilde W_{\alpha,\beta}[\varphi](\nex)\big|\lesssim & ~ h^{P-1}\label{eq:HS_ERROR_VEC}
    \end{align}\label{eq:d2case}\end{subequations}
  hold for the elastostatic and elastodynamic integral operators, where the non-singular approximations  $\widetilde V_{\alpha,\beta}$ and $\widetilde W_{\alpha,\beta}$ are defined in~\eqref{eq:approx_ops}.
\end{theorem}
\begin{proof}
  First,  we  prove the assertion for the Laplace single-layer and hypersingular operators, which are given by
  \begin{align}
    V_{0,-1}[\varphi](\nex) = S[\varphi](\nex)= & -\frac{1}{2 \pi}\int_{\Gamma}\log|\nex-\ney| \varphi(\ney) \mathrm{d} s(\ney)\quad\mbox{and}                                                                                                                                                                                                                                                           \\
    W_{1,0}[\varphi](\nex) = T[\varphi](\nex)=  & ~\frac{1}{2 \pi} {\rm f.p.}\int_{\Gamma}\left\{\frac{\nor(\ney) \cdot \nor(\nex)}{|\nex-\ney|^{2}}-2 \frac{\nor(\ney) \cdot(\nex-\ney)(\nex-\ney) \cdot \nor(\nex)}{|\nex-\ney|^{4}}\right\} \varphi(\ney) \mathrm{d} s(\ney),
  \end{align}
  respectively.  As in Lemma~\ref{lem:d2case}, we let $\chi:[a,b]\to \Gamma_h(\nex)$ denote a local smooth parametrization of $\Gamma_h(\nex)$. Note that by virtue of the injectivity and smoothness of $\chi$, there exist constants $c_\chi$ and $\tilde c_\chi$ such that $\tilde c_\chi h\leq \delta\leq c_\chi h $ with $\delta = b-a$.

  The error in the approximation of the single-layer operator is given by
  \begin{multline}
      (V_{0,-1}-\widetilde V_{0,-1})[\varphi](\nex)
 = -\frac{1}{2\pi}\int_{\Gamma_h(\nex)}\log|\nex-\ney|\lf\{\varphi(\ney)-\gamma_1\Phi(\ney;\nex)\rg\}\de s(\ney) \\
   +\frac{1}{2\pi}\int_{\Gamma_h(\nex)}\frac{(\nex-\ney)\cdot\nor(\ney)}{|\nex-\ney|^2}\gamma_0\Phi(\ney;\nex)\de s(\ney). \label{eq:errror_SL}
  \end{multline}
  Employing the local curve parametrization $\chi$, the first integral can be recast as
  \begin{multline}
      -\frac{1}{2\pi}\int_{\Gamma_h(\nex)}\log|\nex-\ney|\lf\{\varphi(\ney)-\gamma_1\Phi(\ney;\nex)\rg\}\de s(\ney) \\
      =I_0[g_0](t) - \frac{1}{4\pi}\int_{a}^b\log\lf(\frac{|\chi(t)-\chi(s)|^2}{(t-s)^2}\rg)\lf\{\varphi(\chi(s)) - \gamma_1\Phi(\chi(s);\nex)\rg\})|\chi'(s)|\de s ,\label{eq:errror_SL2}
  \end{multline}
  where $\nex=\chi(t)$ and
  $$
    g_0(s) =-\frac{1}{2\pi}|\chi'(s)|\lf\{\varphi(\chi(s)) - \gamma_1\Phi(\chi(s);\chi(t))\rg\}.
  $$
  Since the last integrals in~\eqref{eq:errror_SL} and~\eqref{eq:errror_SL2} can be bounded as $\lesssim  h^{P+1}$ using Lemma~\ref{lem:d2case}, the leading error term in~\eqref{eq:errror_SL}, for small values of $h>0$, is the weakly-singular integral $I_0[g_0](t)$. Now, in view of Lemma~\ref{lem:1D_IO}, we have
  \begin{equation}\label{eq:interm1}
    |I_0[g_0](t)|\leq 2  \delta\{1-\log (\tilde \mu\delta)\}\|g_0\|_{C^{0}[a,b]}\lesssim h |\log h|\|g_0\|_{C^0[a,b]},
  \end{equation}
  for some constant $\tilde\mu>0$ depending on $\mu$ and $c_{\chi}$. The norm of $g_0$ can be estimated using Lemma~\ref{lem:d2case} and the smoothness of $|\chi'|$,  to achieve $\|g_0\|_{C^0[a,b]}\lesssim \delta^{P}\lesssim  h^{P}$. Therefore, from~\eqref{eq:errror_SL},~\eqref{eq:errror_SL2}, and~\eqref{eq:interm1}, we  obtain
  $$\lf|(V_{0,-1}-\widetilde V_{0,-1})[\varphi](\nex)\rg|\lesssim  h^{P+1}|\log h|,$$
  which proves the bound in~\eqref{eq:SL_ERROR} for the Laplace single-layer operator.

  The error in the approximation of the hypersingular operator, on the other hand, is given by
  \begin{align}
      (W_{1,0}-\widetilde W_{1,0})[\varphi](\nex)
      &= -\frac{1}{2 \pi} {\rm f.p.}\int_{\Gamma_h(\nex)}\frac{\nor(\ney) \cdot \nor(\nex)}{|\nex-\ney|^{2}} \lf\{\varphi(\ney)-\gamma_0\Phi(\ney;\nex)\rg\} \mathrm{d} s(\ney) \notag \\
      & \qquad +\frac{1}{\pi} \int_{\Gamma_h(\nex)}\frac{\nor(\ney) \cdot(\nex-\ney)(\nex-\ney) \cdot \nor(\nex)}{|\nex-\ney|^{4}}  \lf\{\varphi(\ney)-\gamma_0\Phi(\ney;\nex)\rg\} \mathrm{d} s(\ney) \notag \\
      & \qquad -\frac{1}{2\pi}\int_{\Gamma_h(\nex)}\frac{(\nex-\ney)\cdot\nor(\nex)}{|\nex-\ney|^2}\gamma_1\Phi(\ney;\nex)\de s(\ney). \label{eq:hyper_bound}
    \end{align}
  The last two integrals in~\eqref{eq:hyper_bound} can be bounded as $\lesssim  h^{P+1}$ using Lemma~\ref{lem:d2case}. Therefore, the leading term in~\eqref{eq:hyper_bound} as $h\to0$ is the finite-part integral, that can be recast as
  \begin{equation}
    -\frac{1}{2 \pi} {\rm f.p.}\int_{\Gamma_h(\nex)}\frac{\nor(\ney) \cdot \nor(\nex)}{|\nex-\ney|^{2}} \lf\{\varphi(\ney)-\gamma_0\Phi(\ney;\nex)\rg\} \mathrm{d} s(\ney)= I_2[g_2](t),
  \end{equation}
  where
  \begin{align}
    g_2(s) = & ~\frac{1}{2\pi}\frac{(t-s)^2}{|\chi(t)-\chi(s)|^{2}}\nor(\chi(s)) \cdot \nor(\chi(t))|\chi'(s)|\lf\{\varphi(\chi(s)) - \Phi(\chi(s);\chi(t))\rg\}.
  \end{align}
  From Lemma~\ref{lem:1D_IO}, we have
  \begin{equation}\label{eq:interm2}
    |I_2[g_2](t)|\leq \frac{2}{\tilde\mu\delta}\|g_2\|_{C^0[a,b]}+\delta|g_2|_{C^{1,1}[a,b]}-2\log(\tilde\mu)\|g_2\|_{C^1[a,b]},
  \end{equation}
  where using the result of Lemma~\ref{lem:d2case}, the norms and semi-norm of $g_2$ can be estimated as
  \begin{equation}\label{eq:interm3}
    \|g_2\|_{C^0[a,b]}\lesssim  h^P,\quad |g_2|_{C^{1,1}[a,b]}\leq \|g_2\|_{C^2[a,b]}\lesssim  h^{P-2}, \andtext \|g_2\|_{C^1[a,b]}\lesssim  h^{P-1}.
  \end{equation}
  Combining~\eqref{eq:hyper_bound},~\eqref{eq:interm2}, and~\eqref{eq:interm3}, we thus arrive at
  $$\lf|(W_{1,0}-\widetilde W_{1,0})[\varphi](\nex)\rg|\lesssim  h^{P-1},$$
  which proves the bound in~\eqref{eq:HS_ERROR} for the Laplace hypersingular operator.

The proof  for the remaining integral operator can be performed by taking into account the leading kernel singularity associated to each kernel, which are summarized in Table~\ref{tab:singularity-types}. In addition, the (Helmholtz, elastodynamic) frequency-domain kernels have the same singular part as their zero-frequency counterpart, i.e. the kernel differences are bounded, so that estimates found for the zero-frequency singular operators carry over to their frequency-domain analogs.
\end{proof}

Finally, we comment on how this analysis can be extended to the $d=3$ case, corresponding to surfaces in three dimensions.  First, Lemma~\ref{lem:1D_IO}  can be used---through a change of variables to polar coordinates centered at $t= \chi^{-1}(\nex)\subset\R^2$ in the parameter space---to estimate the errors in the approximations~\eqref{eq:approx_ops} in terms of the size $h>0$ of $\Gamma_h(\nex)$ and the norm of the interpolation errors functions~\eqref{eq:functions_to_estimate} over $\Gamma_h(\nex)$. Then, an interpolation result akin to Lemma~\ref{lem:d2case} is needed.  On this regards we distinguish between surface discretizations based on quadrilateral and triangular geometric patches. In the former setting, an analysis similar to the one carried out above could  be performed to derive error estimates of the form
\begin{equation}\label{eq:bounds_3D}
  |\p^{\theta}e_\alpha(s)|\lesssim h^{p-|\theta|}\andtext|\p^{\theta}e_\beta(s)|\lesssim h^{p-|\theta|},
\end{equation}
for the interpolation errors functions~\eqref{eq:functions_to_estimate}, where $s=(s_1,s_2)$, $\theta=(\theta_1,\theta_2)$, and $|\theta|=\theta_1+\theta_2$,  for sets of $P=p^2$ collocation points given by tensor products of one dimensional grids consisting of $ p$ points per dimension. In the latter setting, in turn, results on Lagrange interpolation over triangles~(e.g.,~\cite[Sec. 5.1.1]{atkinson1997numerical}) suggest that choosing a total of $P=p(p+1)/2$ collocation points inside the triangle, in such a way that they uniquely determine a Lagrange interpolation polynomial of the form $L(\xi) = \sum_{0\leq|\theta|\leq p-1}c_{\theta}\xi^{\theta}$ in the parameter space, one could achieve interpolation error bounds such as~\eqref{eq:bounds_3D} with $p =(\sqrt{8P+1}-1)/2$. Finally, assuming the error estimates~\eqref{eq:bounds_3D} hold for the three-dimensional collocation density interpolant and following the arguments in the proof of Theorem~\ref{tm:d2case}, we arrive at the following error estimates for the three-dimensional Laplace and Helmholtz integral operators
\begin{subequations}\begin{align}
      \lf|V_{\alpha,\beta}[\varphi](\nex)-\widetilde V_{\alpha,\beta}[\varphi](\nex)\rg|\lesssim & ~h^{p+1}\quad\mbox{and}\label{eq:SL_ERROR_3D} \\
      \lf|W_{\alpha,\beta}[\varphi](\nex)-\widetilde W_{\alpha,\beta}[\varphi](\nex)\rg|\lesssim & ~h^{p-1},\label{eq:HS_ERROR_3D}
    \end{align}
and the following for the elastostatic and elastodynamic integral operators
    \begin{align}
      \lf|V_{\alpha,\beta}[\varphi](\nex)-\widetilde V_{\alpha,\beta}[\varphi](\nex)\rg|\lesssim & ~ h^{p}\quad\mbox{and}\label{eq:SL_ERROR_VEC_3D}  \\
      \lf|W_{\alpha,\beta}[\varphi](\nex)-\widetilde W_{\alpha,\beta}[\varphi](\nex)\rg|\lesssim & ~ h^{p-1},\label{eq:HS_ERROR_VEC_3D}
\end{align}\label{eq:d3case}\end{subequations}
where the parameter $p$ depends on the total number of quadrature/collocation points $P$, and on the shape of the geometric patches used in the discretization of the surface~$\Gamma$.

The results shown thereafter in Figures~\ref{fig:convergence_SLDL} and~\ref{fig:convergence_ADLHS} demonstrate that the error estimates in~\eqref{eq:d2case} and in~\eqref{eq:d3case} are in fact achieved in practice.\enlargethispage*{5ex}

\section{Numerical examples}\label{sec:numerics}

In this section we present a variety of numerical examples designed to validate the proposed high-order kernel regularization procedure, as well as demonstrate its capability to treat two- and three-dimensional problems of either scalar or vector nature.
\begin{figure}[b]
  \begin{subfigure}{1\linewidth}
  \centering
  \includegraphics[width=0.49\textwidth]{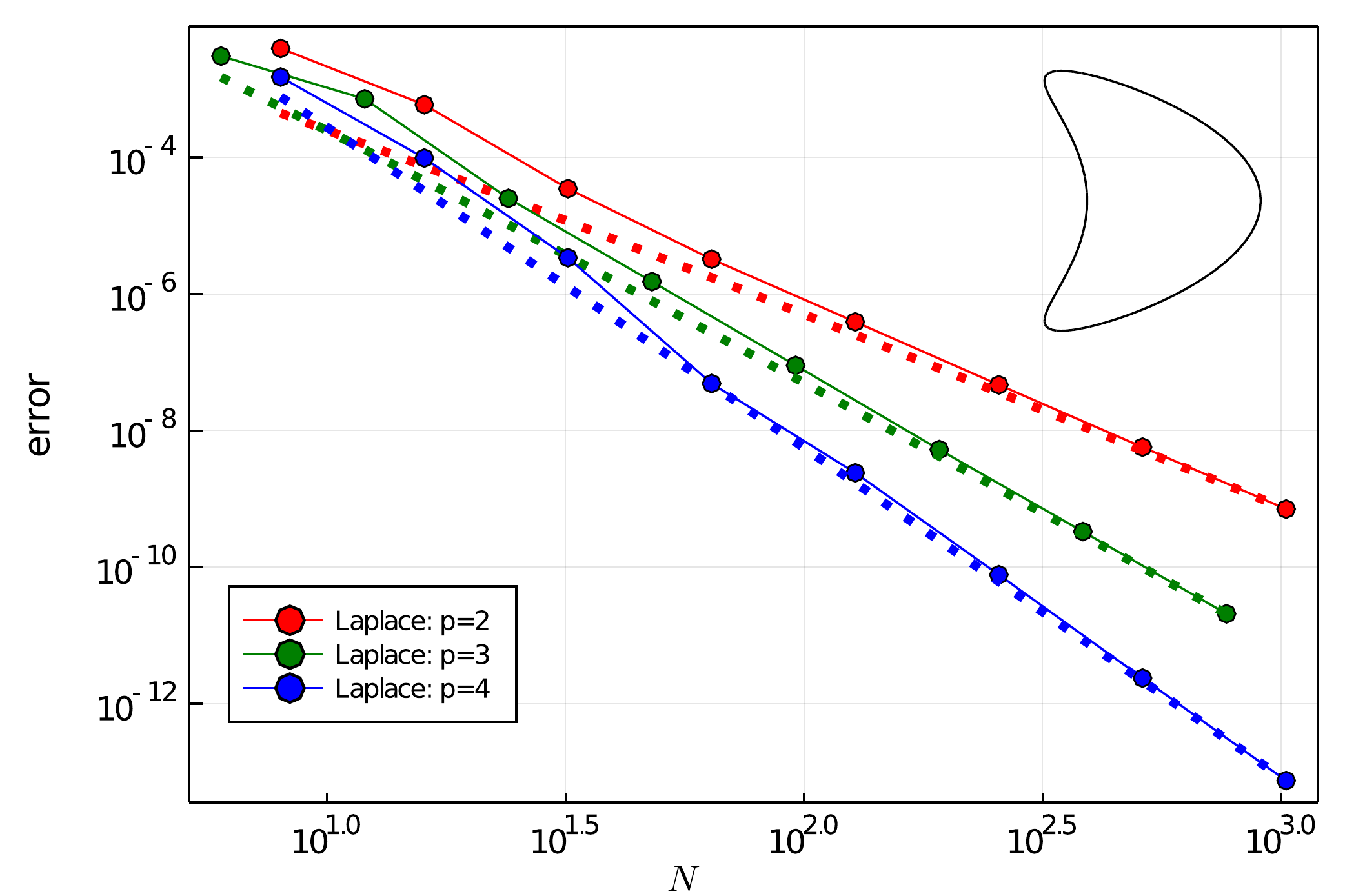}\ \ \includegraphics[width=0.49\textwidth]{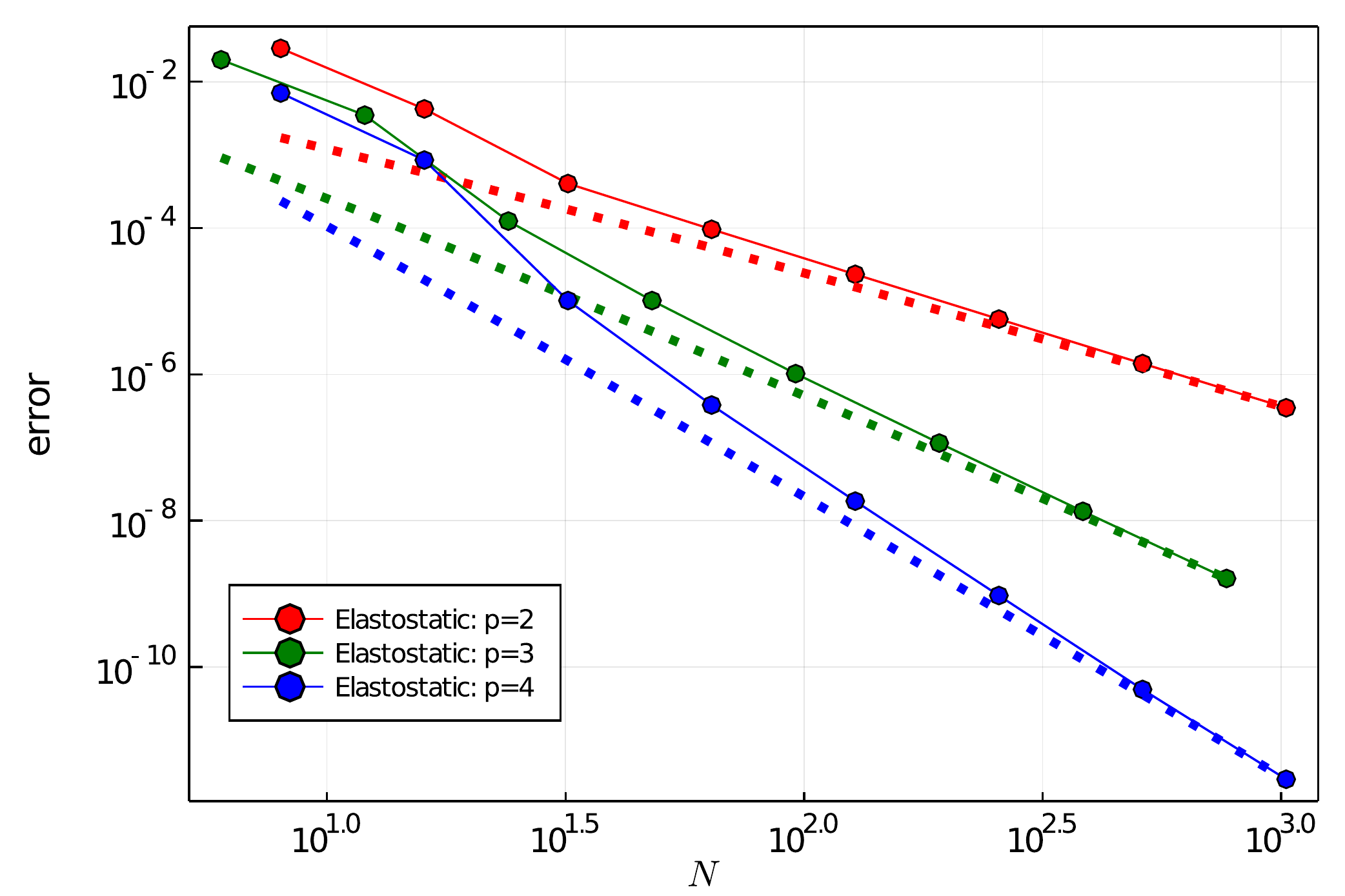}
       \caption{Convergence in two dimensions. The dotted lines indicate orders ${p\!+\!1}$ for the Laplace operator (left) and $p$ for the elastostatic operator (right). Up to logarithmic terms, the observed convergence orders agree with estimates~\eqref{eq:SL_ERROR} and \eqref{eq:SL_ERROR_VEC} in Theorem~\ref{tm:d2case}, with $h\propto N^{-1}$.}
  \label{fig:convergence-greens-2d}
  \end{subfigure}
  \begin{subfigure}{1\linewidth}
  \includegraphics[width=0.49\textwidth]{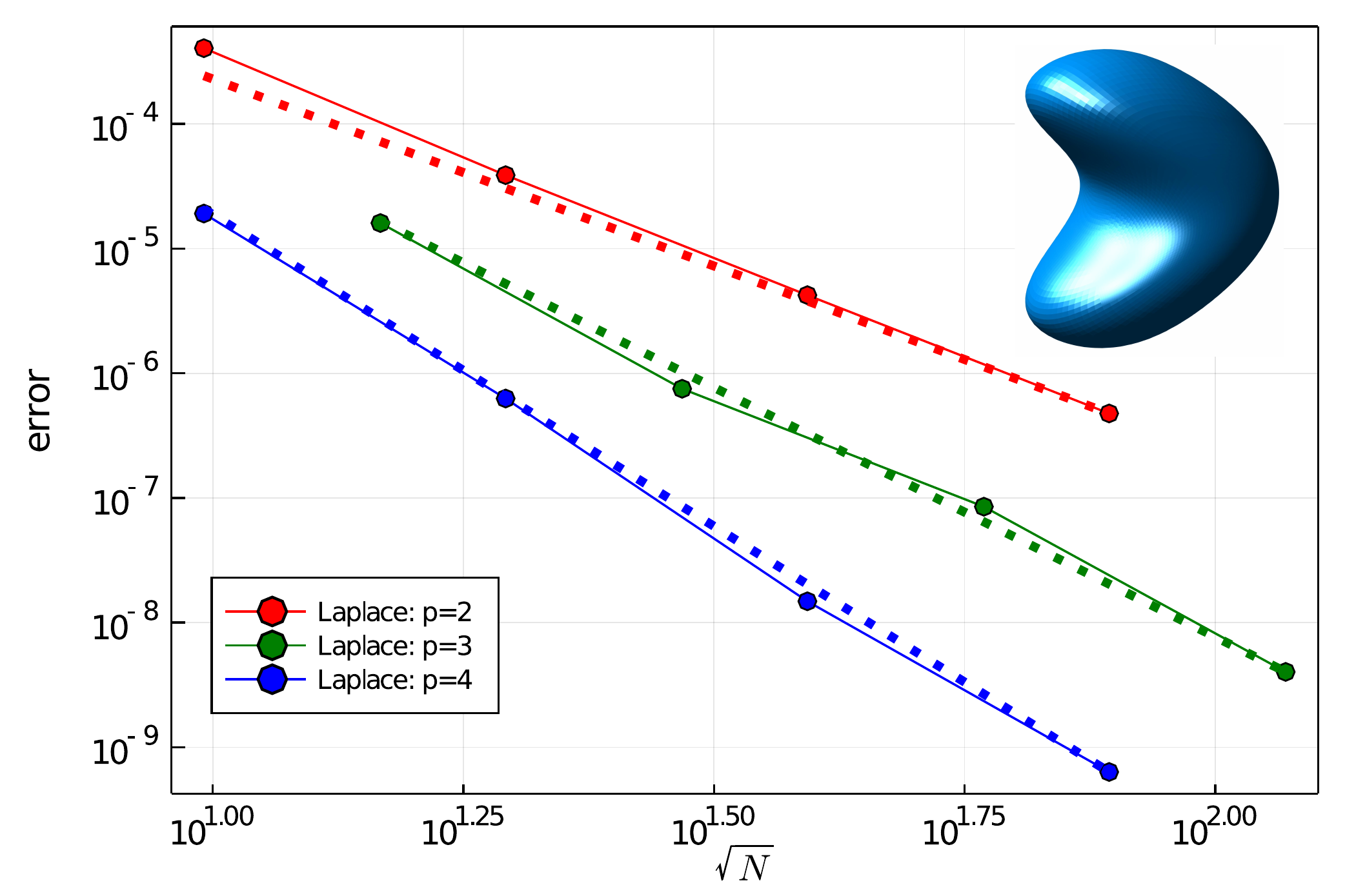}\ \ \includegraphics[width=0.49\textwidth]{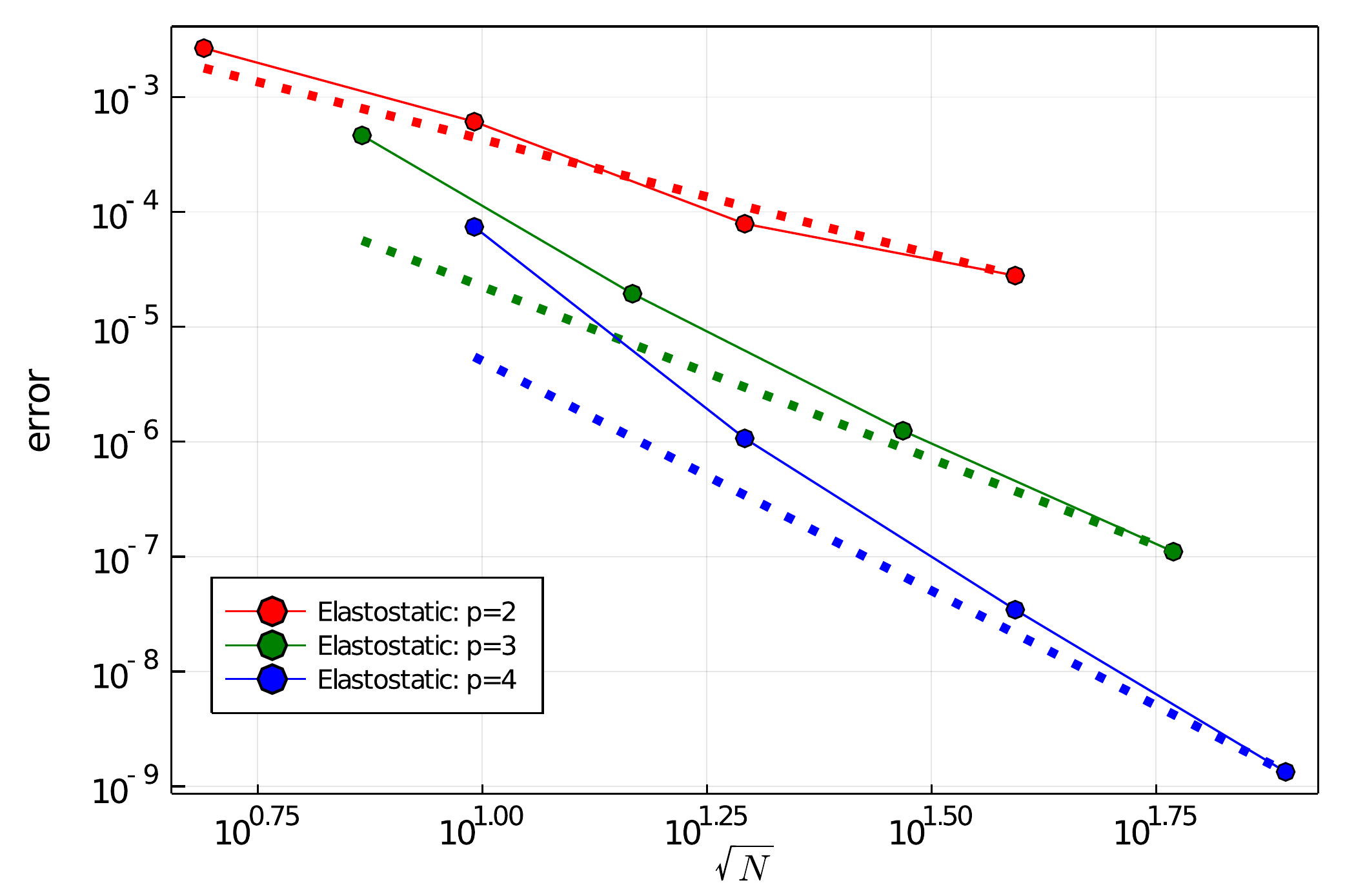}
\caption{Convergence in three dimensions. The dotted lines indicate orders ${p\!+\!1}$ for the Laplace operator (left) and $p$ for the elastostatic operator (right). The observed convergence orders agree with estimates~\eqref{eq:SL_ERROR_3D} and \eqref{eq:SL_ERROR_VEC_3D}, with $h\propto N^{-1/2}$.}
  \label{fig:convergence-greens-3d}
   \end{subfigure}
\caption{Numerical errors ($E_1$ in~\eqref{eq:errors_vec}) in the evaluation of Green's formula~\eqref{eq:greens-formula-1} for the Laplace (left) and elastostatic (right) single- and double-layer operators in two dimensions (a) and three dimensions (b). The results are shown for three
    different values of $p$ (with $p$  denoting the number of Gauss-Legendre quadrature/interpolation nodes per patch per dimension) and various total numbers $N$ of quadrature nodes. The reference slopes, shown as dotted lines, match the convergence orders established in Section~\ref{sec:error_analysis}.}\label{fig:convergence_SLDL}
\end{figure}

\subsection{Operator evaluation}\label{sec:conv-greens-ident}

In this first set of examples we consider the errors incurred in the evaluation of the on-surface Green's identities~\eqref{eq:greens-formula-1} and \eqref{eq:greens-formula-2}, when all four integral operators $S$, $K$, $K'$, and $T$, are approximated by the procedures presented in Section~\ref{sec:numer-discr}. These examples serve as an initial validation of the proposed methodology, for the PDEs~\eqref{eq:considered-PDEs} in both two and three spatial dimensions, as they allow us to easily measure the on-surface errors in the evaluation of the integral operators on given densities.

Throughout the present section we take~$\Gamma \subset \R^d$ as the kite-shaped curve~\cite[p.~79]{COLTON:2012} for $d=2$ (see the inset in Figure~\ref{fig:convergence-greens-2d}) and the bean-shaped surface~\cite[p.~104]{Bruno:2001ima} for $d=3$ (see the inset in Figure~\ref{fig:convergence-greens-3d}). For each PDE considered we construct an exact solution of the homogeneous PDE in the interior of $\Gamma$, given by $\uref(\ner) = G(\ner,\bx_s)$, where $\bx_s=(1,1)$ and $\bx_s = (1,1,1)$ for $d=2$ and $d=3$, respectively, and where $G$ is the {free-space Green's function} provided in Appendix~\ref{app:free_space_Green}. Since $\bx_s$ lies outside of $\Omega$, the function $\uref$ is a solution inside of $\Omega$ of the associated PDE. This solution is used as reference to assess the errors in the numerically approximated Green's formulae.\enlargethispage*{1ex}

The curve $\Gamma$ in the $d=2$ case, is partitioned into $M$ non-overlapping patches where $p$-point Gauss-Legendre quadrature rules are employed to integrate over each individual patch, so that the same number of quadrature points $P_m=p$ is used for all $m=1,\ldots, M$ (see Section~\ref{sec:quadrature}). In the $d=3$ case, we represent the surface as the union~$M$ non-overlapping logically-quadrilateral patches, as done in~\cite{perez2019planewave}, . We use a tensor-product quadrature rule comprising $p \times p$ Gauss-Legendre nodes to integrate over the patches so that the same number of quadrature points $P_m=p^2$ is used for all $m=1,\ldots,M$. Patches of approximately the same size $h>0$, with $h\propto 1/N$ in the $d=2$ case, and $h\propto 1/\sqrt{N}$ in the $d=3$ case, are used in these examples. This yields a total number of $N = Mp^{d-1}$ quadrature nodes on $\Gamma$, resulting in $N$ (resp. $Nd$) degrees of freedom for scalar (resp. vector) problems since a Nystr\"om discretization is used.

Letting $\mathrm{S}=\mathrm{V}_{0,-1}$, $\mathrm{K}=\mathrm{V}_{1,0}$, $\mathrm{K}'=\mathrm{W}_{0,-1}$, and $\mathrm{T}=\mathrm{W}_{0,1}$, be the matrix approximations of the corresponding operators in \eqref{eq:BIOS}, which are evaluated on given densities as per the procedure described in Section~\ref{sec:forward_map}, we consider two types of errors:\enlargethispage*{5ex}
\begin{align}\label{eq:errors_vec}
  E_1 := \frac{\|\tilde{\bold{u}} - \bold{u}\|_{\infty}}{\|\bold{u}\|_{\infty}}
  \andtext E_2 := \frac{\|\tilde{\bold{v}} - \bold{v}\|_{\infty}}{\|\bold{v}\|_{\infty}},
\end{align}
where the vectors
\begin{equation}
  \half\tilde{\bold{u}} := \mathrm{S}\bold{u} - \mathrm{K}\bold{v} \andtext
  \half\tilde{\bold{v}} := \mathrm{K}'\bold{u} - \mathrm{T}\bold{v}
\end{equation}
are given in terms of
\begin{equation}
  \bold{u} := \big[\gamma_0\uref(\by_1), \ldots, \gamma_0\uref(\by_N)\big]^T \andtext
  \bold{v} := \big[\gamma_1\uref(\by_1), \ldots, \gamma_1\uref(\by_N)\big]^T,
\end{equation}
with  $\ney_j$, $j=1,\ldots,N,$ denoting the quadratures points on $\Gamma$.

\begin{figure}[t]
  \begin{subfigure}{1\linewidth}
  \centering
  \includegraphics[width=0.49\textwidth]{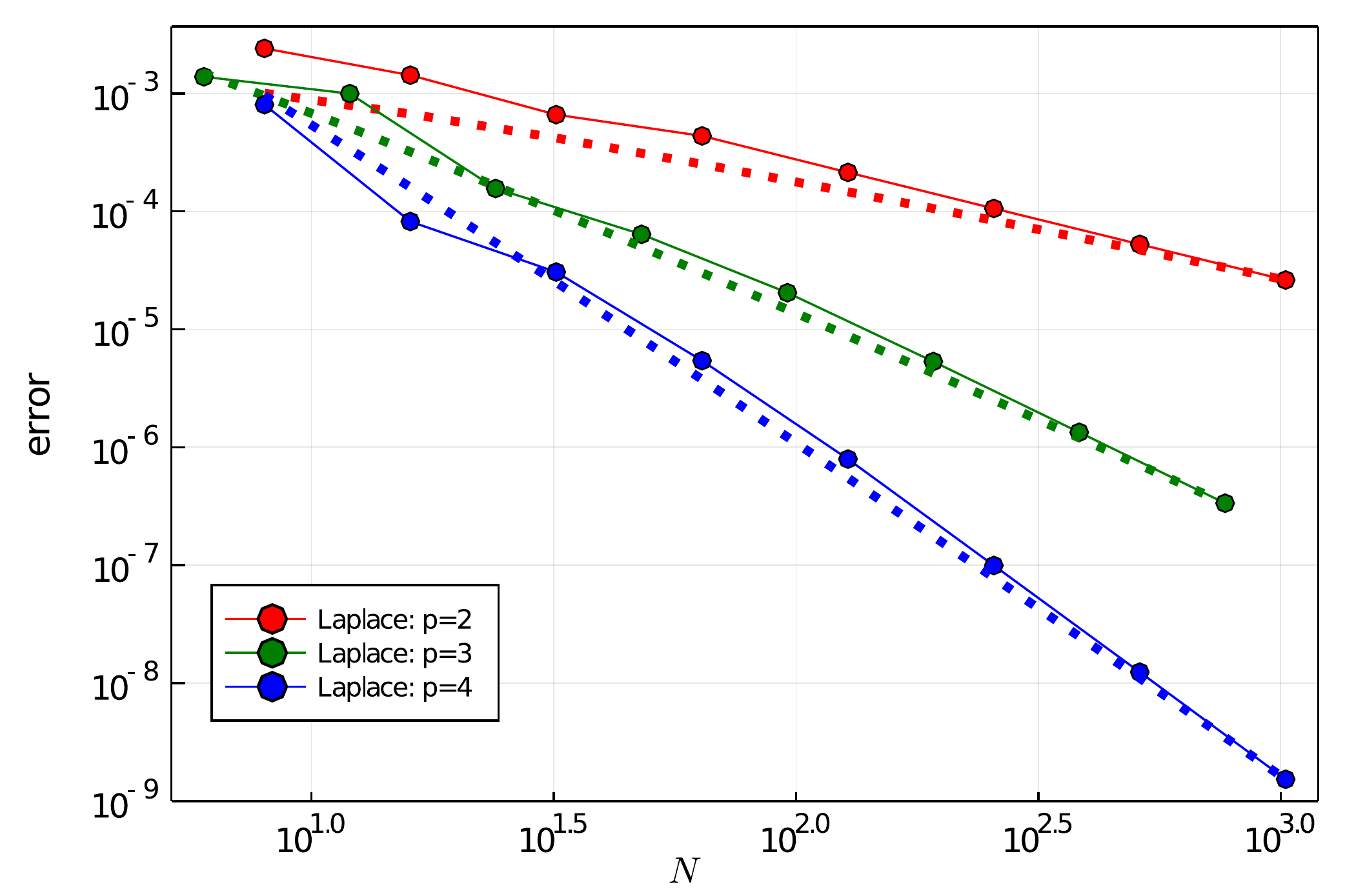}
   \includegraphics[width=0.49\textwidth]{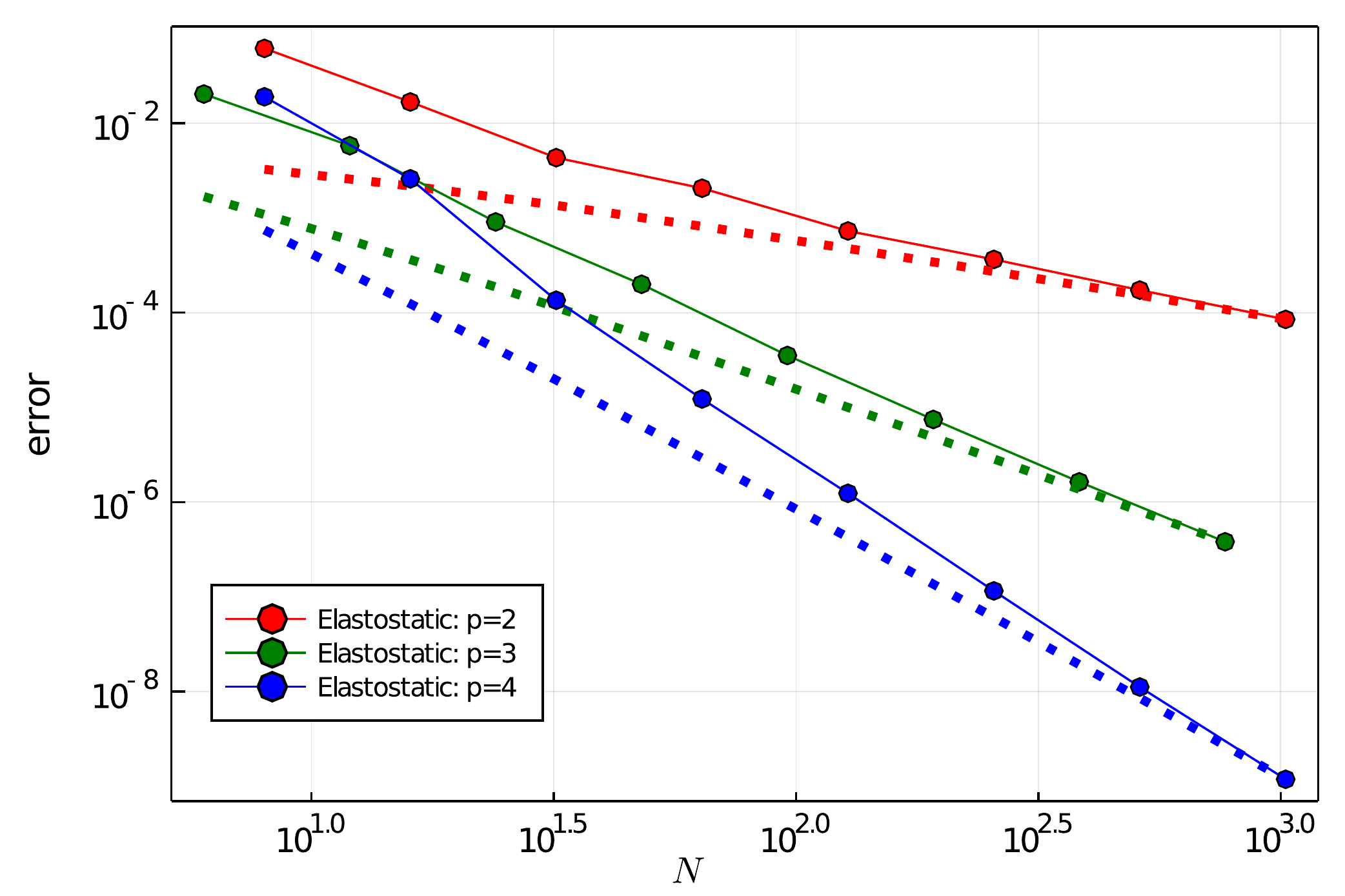}\\
  \caption{Convergence in two dimensions. The dotted lines indicate the order ${p\!-\!1}$ for Laplace (left) and elastostatic (right) operators. The observed convergence orders agree with estimates~\eqref{eq:HS_ERROR} and \eqref{eq:HS_ERROR_VEC} in Theorem~\ref{tm:d2case}, with $h\propto N^{-1}$.}
  \label{fig:convergence-greensp-2d}
  \end{subfigure}
    \begin{subfigure}{1\linewidth}
  \centering
  \includegraphics[width=0.49\textwidth]{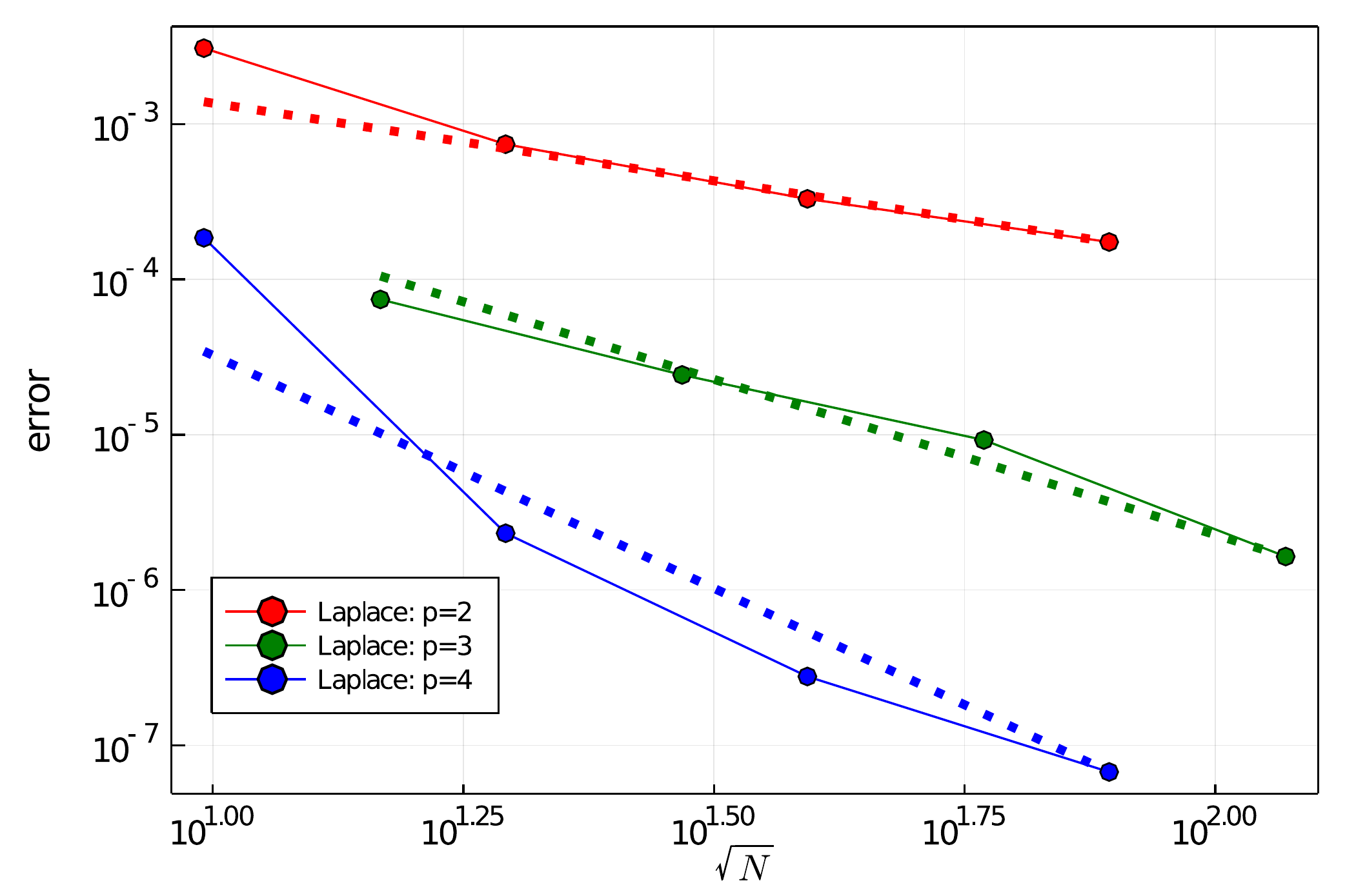}
  \includegraphics[width=0.49\textwidth]{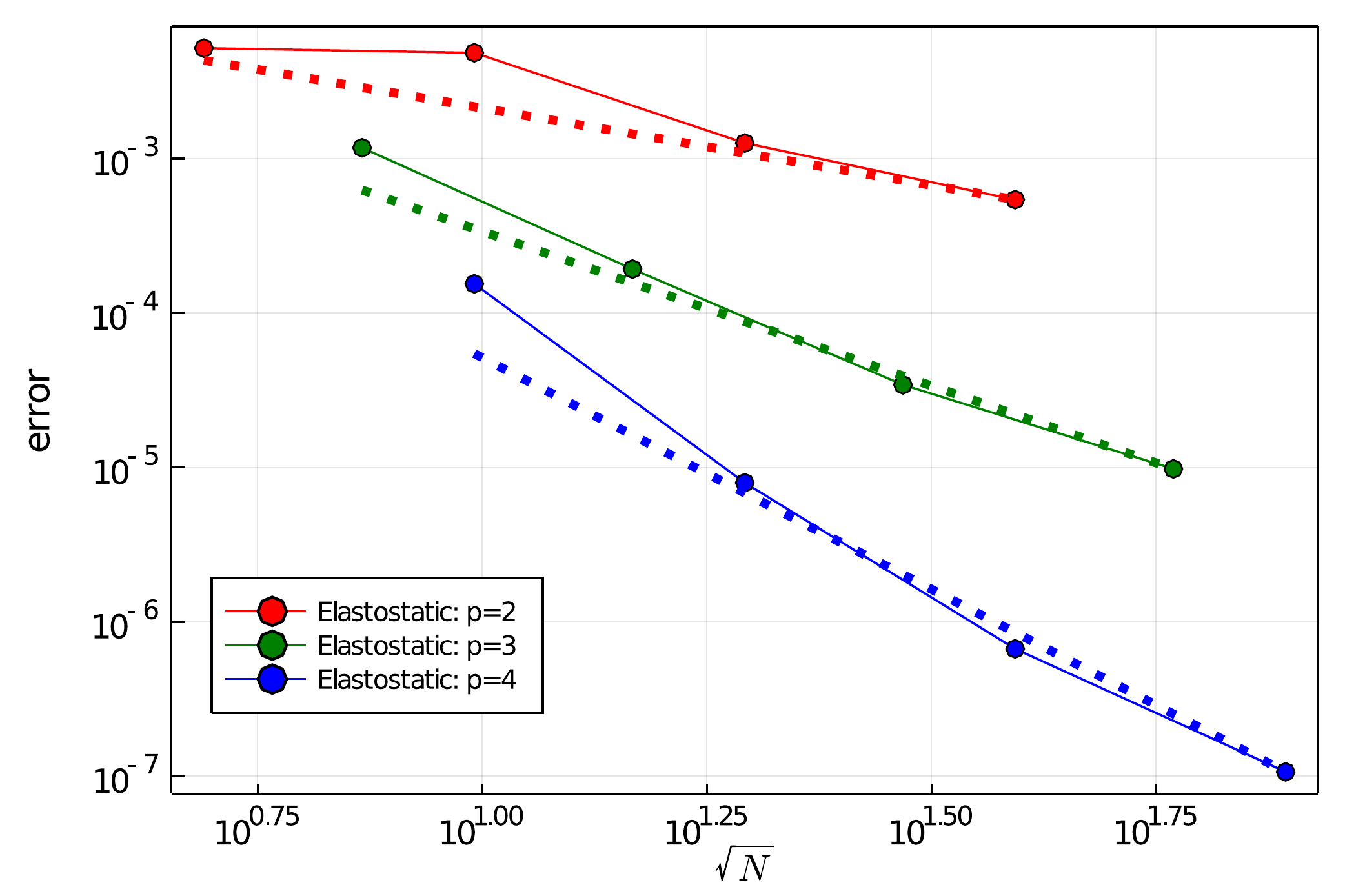}
  \caption{Convergence in three dimensions. The dotted lines indicate the order ${p\!-\!1}$ for Laplace (left) and elastostatic (right) operators. The observed convergence orders agree with estimates~\eqref{eq:HS_ERROR_3D} and \eqref{eq:HS_ERROR_VEC_3D}, with $h\propto N^{-1/2}$.}
  \label{fig:convergence-greensp-3d}
  \end{subfigure}
\caption{Numerical errors ($E_2$ in~\eqref{eq:errors_vec}) in the evaluation of Green's formula~\eqref{eq:greens-formula-2} for the Laplace (left) and elastostatic (right) adjoint double-layer and hypersingular operators in two (a) and three (b) dimensions. The results are shown for three
    different values of $p$ (with $p$  denoting the number of Gauss-Legendre quadrature/interpolation nodes per patch per dimension) and various total numbers $N$ of quadrature nodes. The reference slopes, shown as dotted lines, match the convergence orders established in Section~\ref{sec:error_analysis}.}\label{fig:convergence_ADLHS}
\end{figure}

The $E_1$ errors, corresponding to the combined evaluation of the single- and double-layer operators, are shown in Figure~\ref{fig:convergence-greens-2d} (resp.~\ref{fig:convergence-greens-3d}) for $d=2$ (resp. $d=3$) for various discretization sizes $h\propto N^{1/(d-1)}$ and numbers $p$ of quadrature points per dimension per patch. As discussed in Section~\ref{sec:error_analysis}, for a $p$-point (resp. $(p\times p)$-point) quadrature rule for $d=2$ (resp. $d=3$), we observe,  up to logarithmic factors, $E_1=\Ocal(h^{p+1})$ for the scalar problems and  $E_1=\Ocal(h^{p})$ errors for the vector problems. The difference in the convergence order between the scalar and vector PDEs is due to fact that the double-layer operator is of Cauchy principal value type in the latter case but not in the former, see Table~\ref{tab:singularity-types}. Then, figures~\ref{fig:convergence-greensp-2d} and~\ref{fig:convergence-greensp-3d} display the $E_2$ errors for $d=2$ and $d=3$, respectively, corresponding to the combined evaluation of the adjoint double-layer and hypersingular operators. As discussed in Section~\ref{sec:error_analysis}, the dominant errors in this case, stem from the approximation of the hypersingular operator, yielding $E_2=\Ocal(h^{p-1})$ in both two and three dimensions.

\subsection{Solution of boundary value problems}\label{sec:bvp}

In our next set of examples we apply the proposed methodology to the solution of exterior boundary value problems in the unbounded domain $\R^3\setminus\overline{\Omega}$ of boundary $\Gamma$, focusing on the (scalar) Helmholtz equation for a fixed wavenumber   $k = \omega / c =\pi$. The field $u^{\rm ref}(\ner) := G(\ner,\bold 0)$, generated by a unit point source at the origin, solves (for example) the exterior Dirichlet problem
\begin{equation}
  \Lcal u = 0 \text{ \ in }\R^3\setminus\overline{\Omega}, \quad \gamma_0 u=\gamma_0 u^{\rm ref}\text{ \ on }\Gamma.
\end{equation}
We numerically solve that problem by seeking $u$ as the combined single- and double-layer potential
\begin{equation}
  \label{eq:combined-field-representation}
  u(\br) = (\Dcal-i\omega\Scal)[\varphi](\br), \quad \ner\in\R^d\setminus\Gamma,
\end{equation}
which leads to the combined field integral equation (CFIE):
\begin{equation}
  \label{eq:CFIE} \frac{\varphi(\bx)}{2} + (K-i\omega S)[\varphi](\bx)=f(\bx) ,\quad \bx\in\Gamma,
\end{equation}
for the unknown density function $\varphi:\Gamma\to\C$, where $f:=\gamma_0u^{\rm ref}$. As is well-known, the CFIE admits a unique solution for all frequencies~\cite{COLTON:1983}.

The integral equation~\eqref{eq:CFIE} is discretized using the procedures presented in Section~\ref{sec:numer-discr} and iteratively solved by means of GMRES~\cite{saad1986gmres}.  The resulting numerical errors are measured by
\begin{align}\label{eq:error_ff}
  E^{\rm far} := \frac{\displaystyle \max_{j=1,\ldots,100}\big| \tilde u(\ner_j) - \uref(\ner_j)\big|}{\displaystyle\max_{j=1,\ldots,100}\big|\uref(\ner_j)\big|},
\end{align}
where the numerical solution $u$ is evaluated using~\eqref{eq:combined-field-representation} with $\varphi$ solving~\eqref{eq:CFIE} and the test points $\{\ner_j\}_{j=1}^{100}$ lie on a circle (resp. a sphere) of radius~5 enclosing~$\Gamma$ for $d=2$ (resp. $d=3$). Figure~\ref{fig:2d_kite_ff}, shows the two-dimensional convergence results, where $\Gamma$ is the kite-shaped curve previously used in Section~\ref{sec:conv-greens-ident}. The linear system solutions in these examples were obtained after about $22$ GMRES iterations for a relative error tolerance of $10^{-12}$.  No significant differences in the iteration count were observed between the various $p$ values used. Likewise, Figure~\ref{fig:3d_acorn_ff} displays the solution errors in the three-dimensional case, with $\Gamma$ taken as an acorn-shaped surface. Again, all simulations converged to a residual tolerance of $10^{-12}$ within about $30$ iterations, with no significant differences observed between different values of $p$.\enlargethispage*{1ex}
\begin{figure}[h!]
  \begin{subfigure}{0.5\linewidth}
  \centering
  \includegraphics[width=1\textwidth]{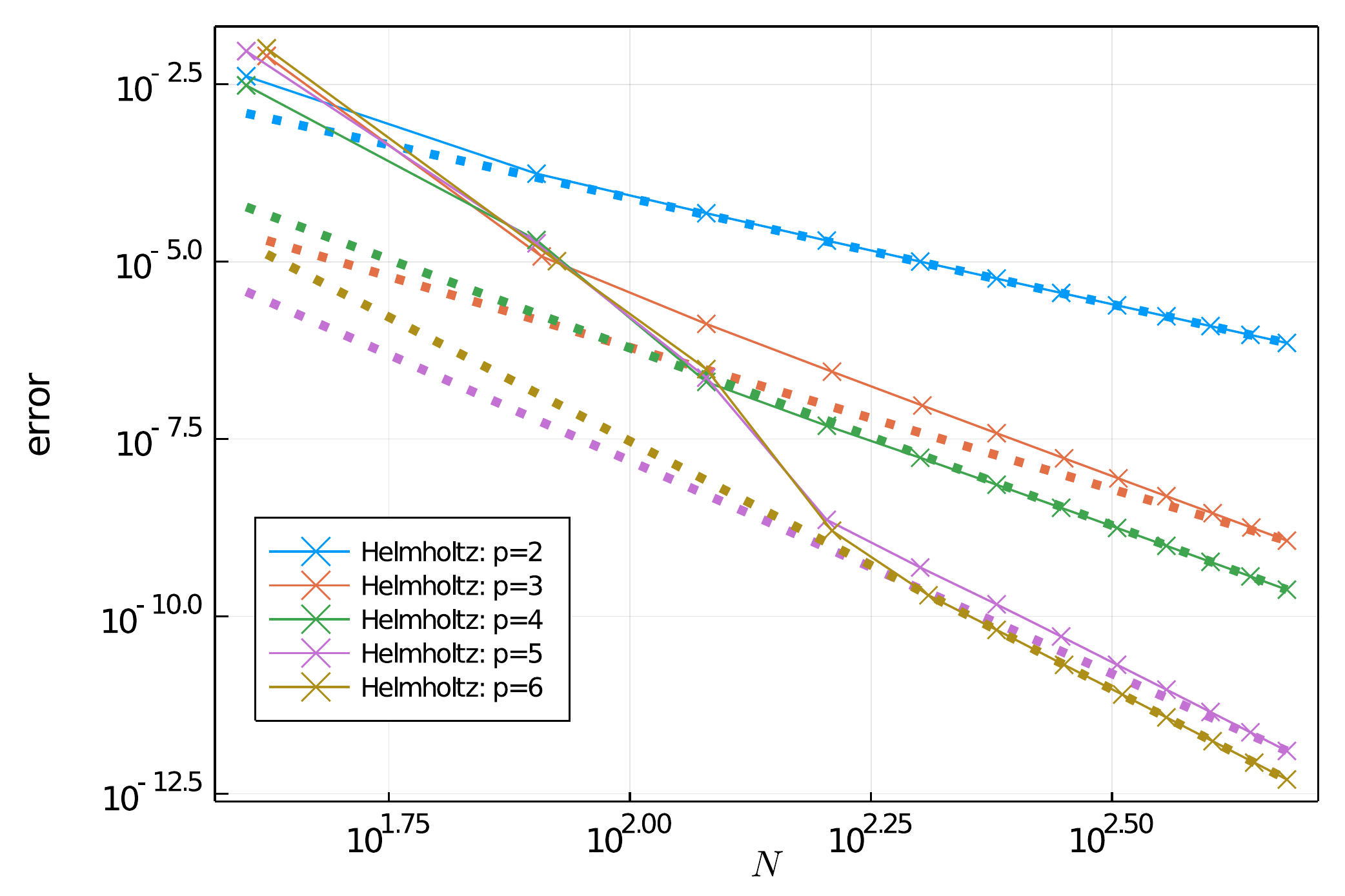}
  \caption{Field errors $E^{\rm far}$ (see~\eqref{eq:error_ff}) in the solution of the Helmholtz equation in two dimensions.\label{fig:2d_kite_ff}}
    \end{subfigure}\qquad
      \begin{subfigure}{0.45\linewidth}  \centering
        \includegraphics[width=0.95\textwidth]{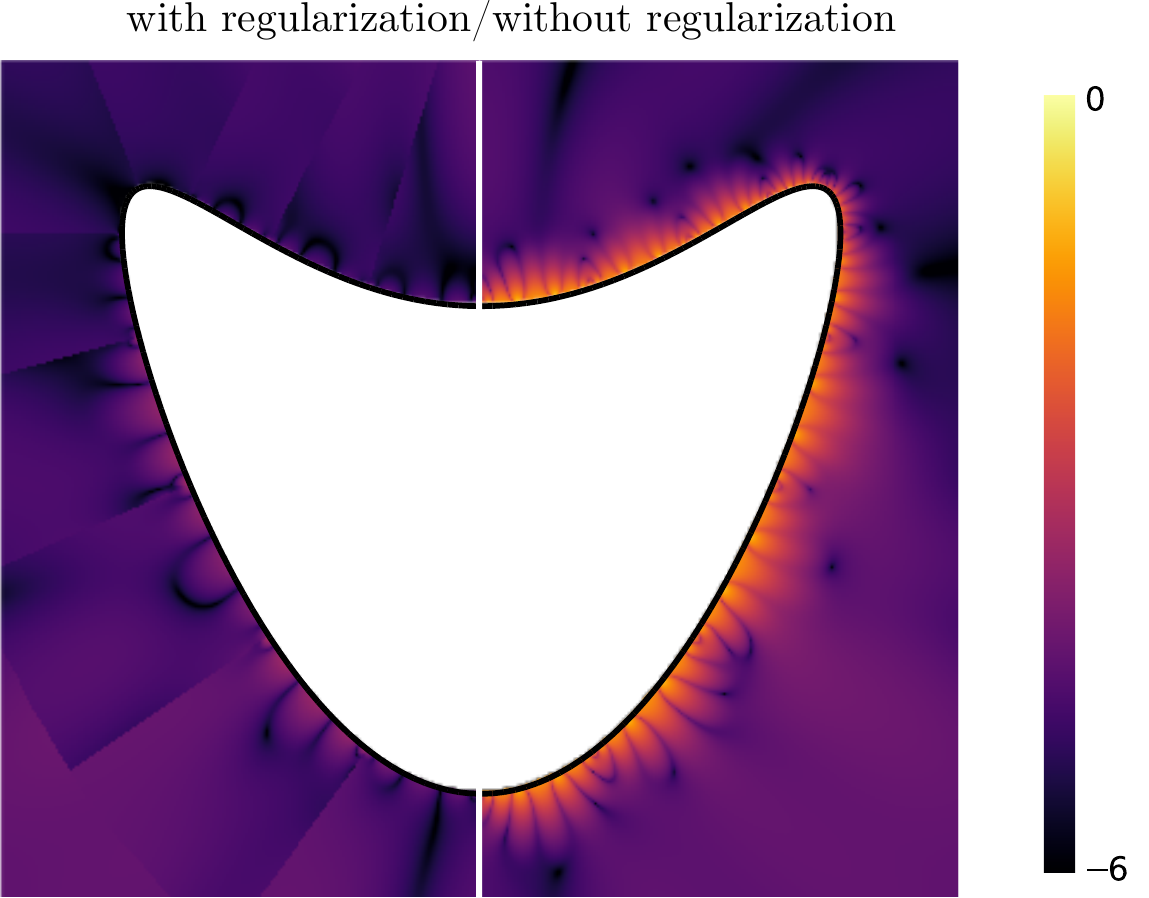}
  \caption{Near-field errors in two dimensions.\label{fig:kite_nf}}
  \end{subfigure}\\
    \begin{subfigure}{0.5\linewidth}
  \centering
\includegraphics[width=1\textwidth]{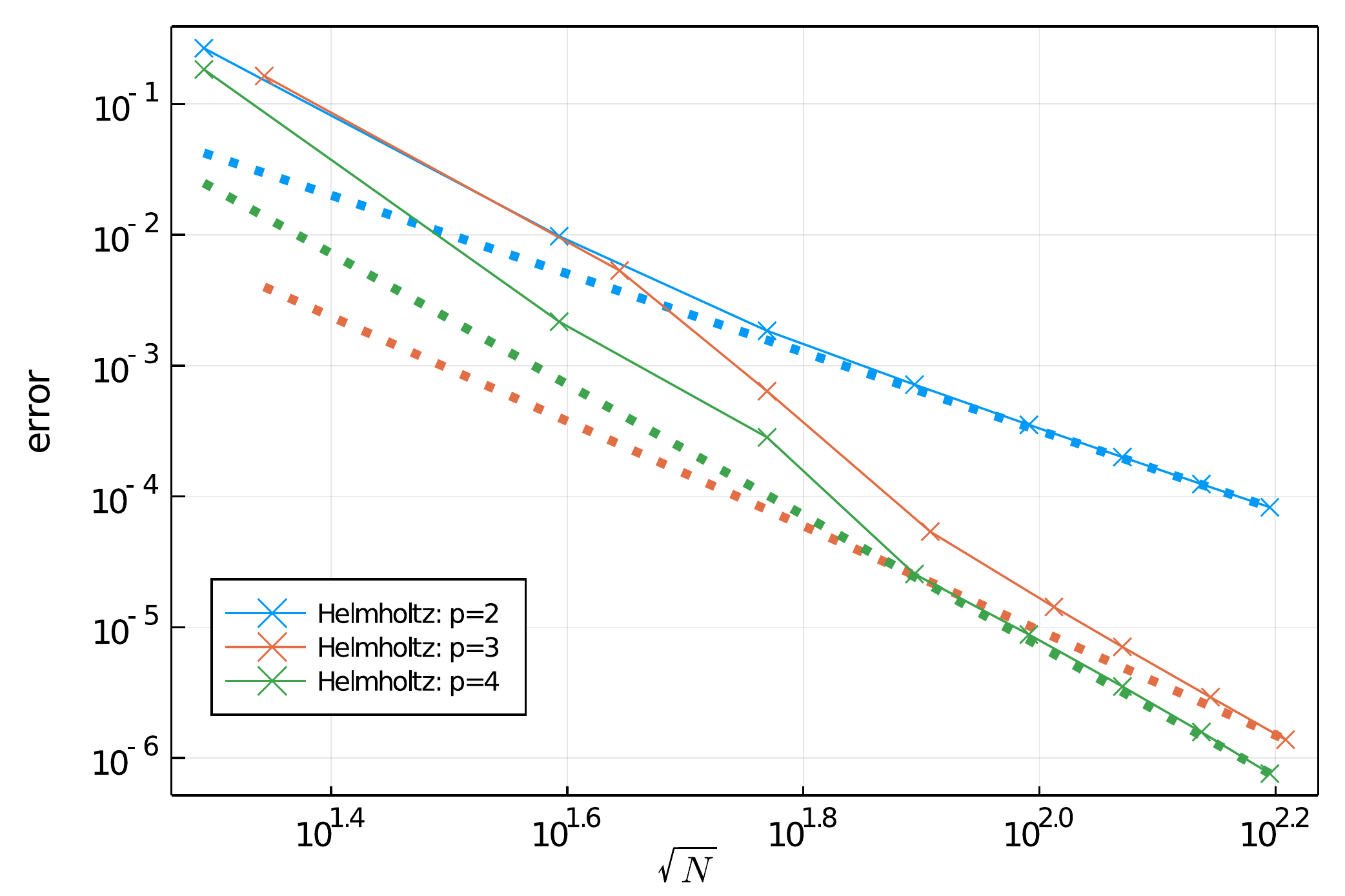}
  \caption{Field errors $E^{\rm far}$ (see~\eqref{eq:error_ff}) in the solution of the Helmholtz equation in three dimensions.\label{fig:3d_acorn_ff}}
  \end{subfigure}
    \begin{subfigure}{0.5\linewidth}
  \centering
        \includegraphics[width=1\textwidth]{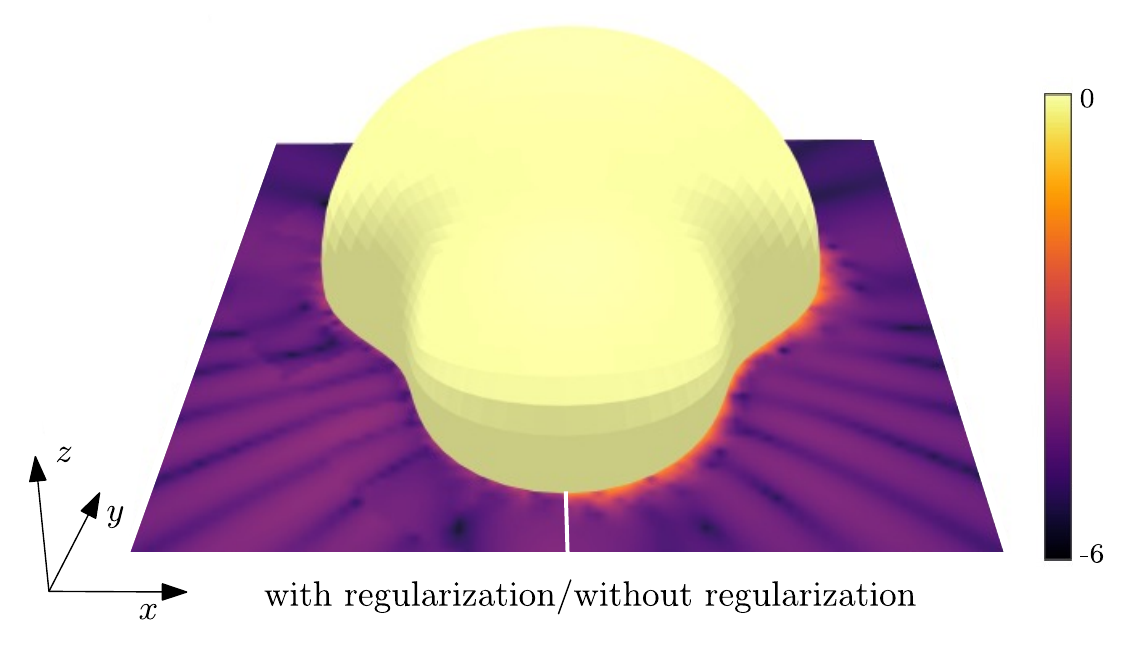}
  \caption{Near-field errors in three-dimensions.\label{fig:acorn_nf}}
  \end{subfigure}
  \caption{Solution of two- and three-dimensional exterior Dirichlet problems using the combined-field integral equation~\eqref{eq:CFIE}. (a) and (c): Convergence plots as the mesh size is refined for various values of $p$, where $p$ denotes the number of quadrature nodes per patch per dimension. The dotted lines represent reference slopes of order $p+1$. The right figure shows the field error $\log_{10}|\uref - \tilde{u}|$ with and without the near-field regularization.}
  \label{fig:exterior-dirichlet-pointsource-2d}
\end{figure}

As discussed in Section~\ref{sec:dens-interp}, it is also possible to regularize the nearly singular integrals that occur when evaluating the potential \eqref{eq:combined-field-representation} at observation points $\br\in\R^d\setminus\Gamma$ located near, but not on, $\Gamma$. The usefulness of this regularization is illustrated in Figures~\ref{fig:kite_nf} and~\ref{fig:acorn_nf} for $d=2$ and $d=3$, respectively, where the absolute solution error $\log_{10}|\uref - \tilde{u}|$ is plotted with or without regularization of the near-singularity (the regularization being in the former case used for all field points $\ner$ such that $\text{dist}(\ner,\Gamma) < 1$). As can be seen in Figures~\ref{fig:kite_nf} and~\ref{fig:acorn_nf}, significantly better results are obtained when regularization is applied.

\begin{remark}\label{rm:superconvergence-odd-p}
  The convergence of the field errors shown in Figures~\ref{fig:2d_kite_ff} and~\ref{fig:3d_acorn_ff} appears to be one order higher than predicted in Section~\ref{sec:error_analysis} for odd values of $p$. The same phenomenon was observed in \cite{perez2019planewave} in the context of a plane-wave DIM for Helmholtz equation, and in \cite{HDI3D} in the context of the harmonic DIM for the Laplace equation.
\end{remark}

\begin{figure}[t]
  \centering
  \begin{subfigure}{0.49\linewidth}
    \includegraphics[width=1\textwidth]{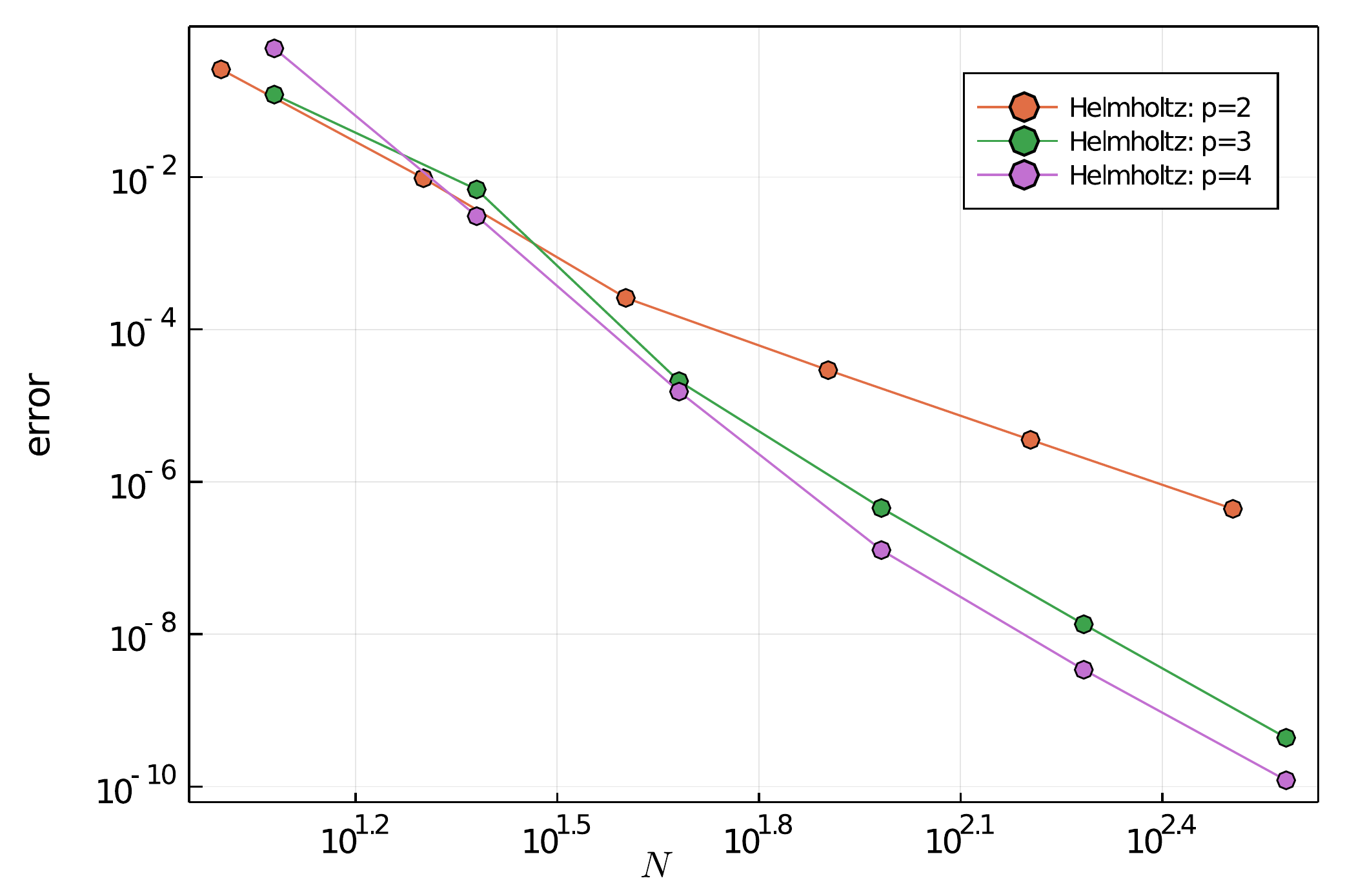}
    \caption{$E^{\rm far}$ $(d=2)$.}
  \end{subfigure}
  \begin{subfigure}{0.49\linewidth}
    \includegraphics[width=1\textwidth]{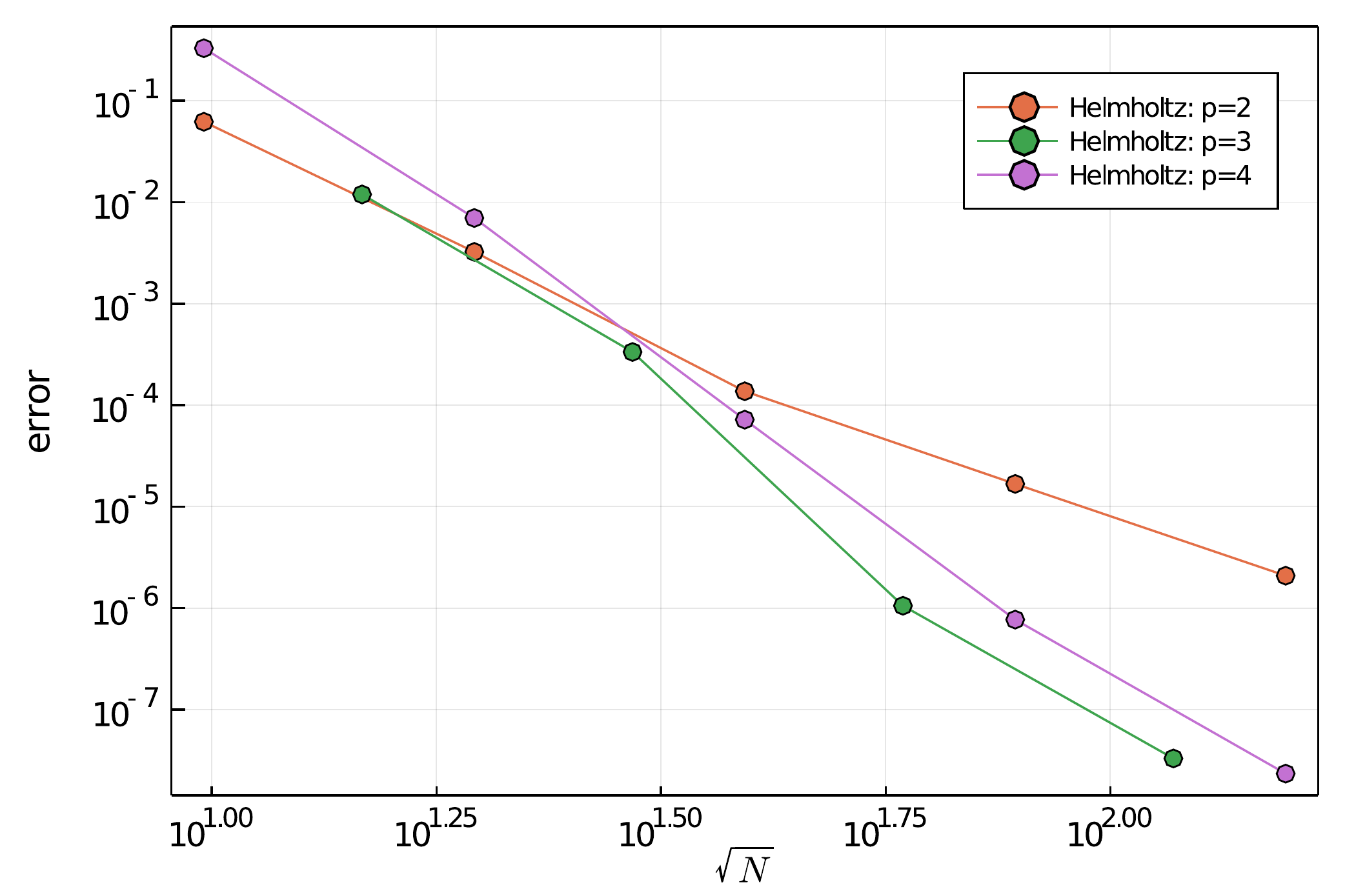}
    \caption{$E^{\rm far}$ $(d=3)$.}
  \end{subfigure}
  \caption{Sound-soft scattering of a plane wave by a unit circle (left) and a unit sphere (right); Helmholtz equation solution with $\omega=2\pi$ obtained by means of the CFIE~\eqref{eq:CFIE}. Relative solution errors $E^{\rm far}$ (see~\eqref{eq:error_ff}) evaluated at a set of evaluation points located on a circle/sphere of radius twice that of the obstacle. The value $p=3$ yields an improved convergence order equal to that for $p=4$ (see Remark~\ref{rm:superconvergence-odd-p}).}
  \label{fig:convergence-scattering-planewave}
\end{figure}

\begin{figure}[t]
  \centering
  \begin{subfigure}{0.49\linewidth}
    \includegraphics[width=1\textwidth]{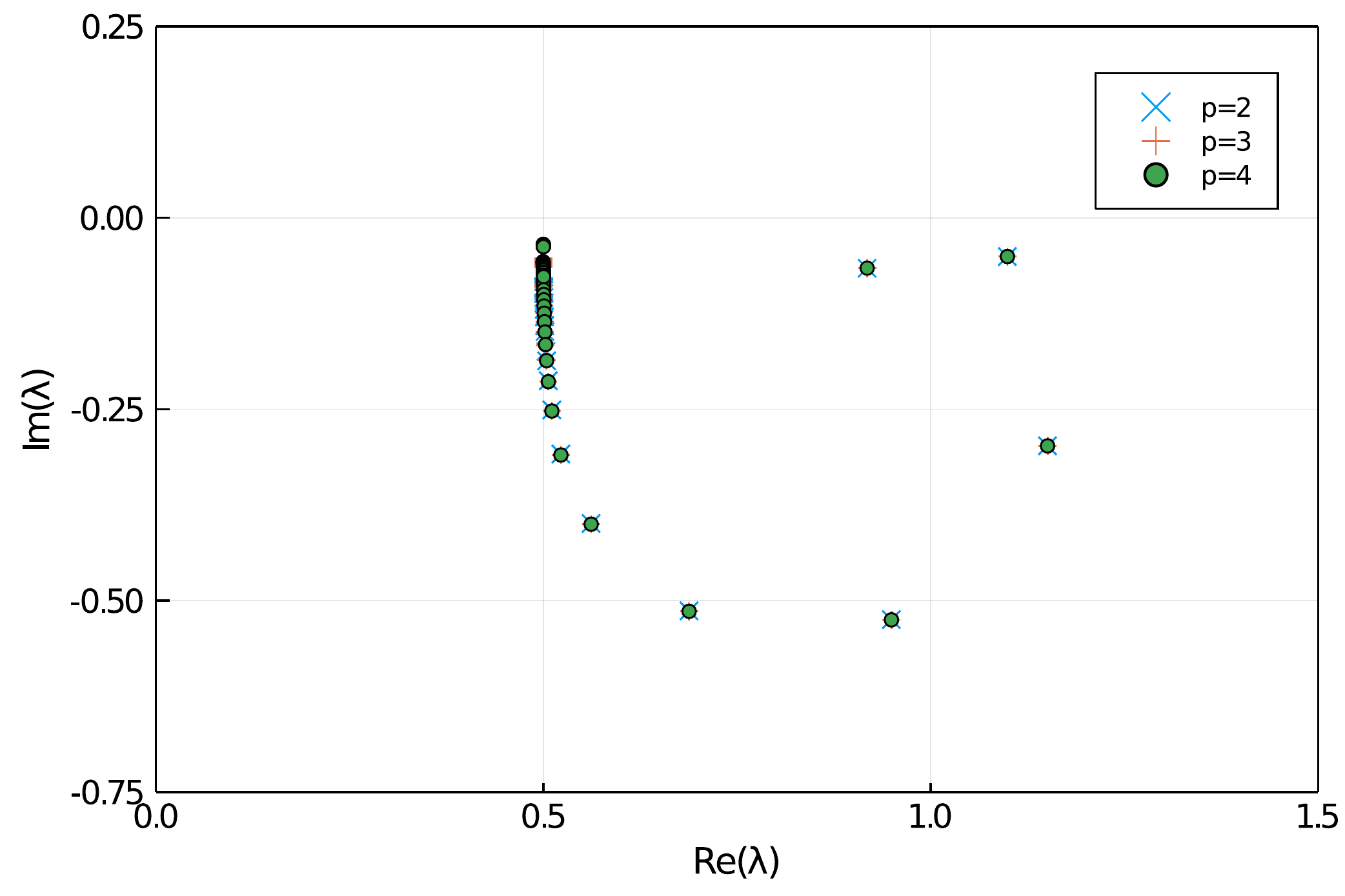}
    \caption{$d=2$.}
    \label{fig:spec-2d}
  \end{subfigure}
  \begin{subfigure}{0.49\linewidth}
    \includegraphics[width=1\textwidth]{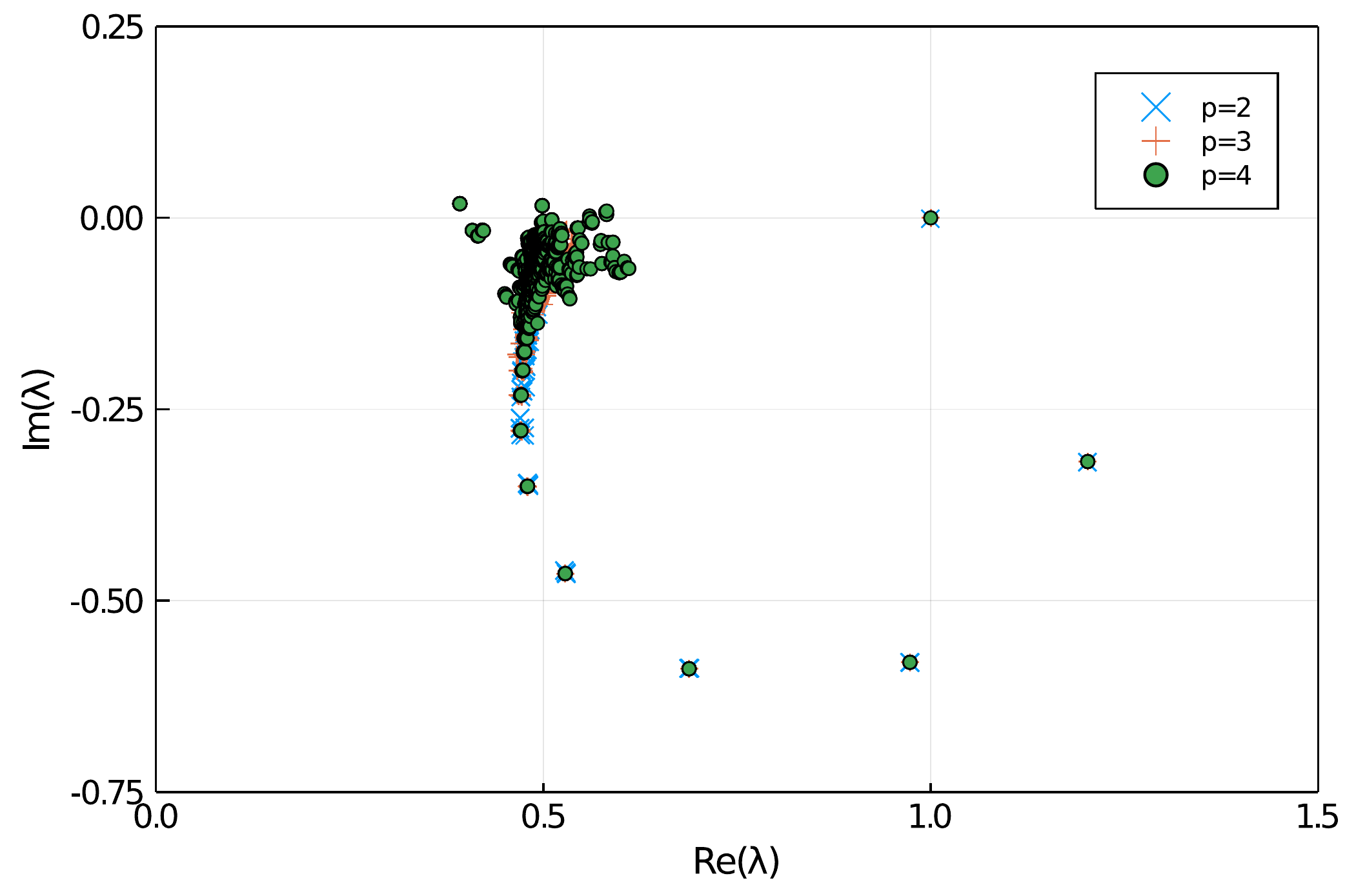}
    \caption{$d=3$.}
      \label{fig:spec-3d}
  \end{subfigure}
  \caption{Sound-soft scattering of a plane wave by a unit circle (left) and a unit sphere (right): spectrum of the CFIE operator.}
  \label{fig:convergence-scattering-planewave}
\end{figure}

In the next example we solve the classical problem of scattering of an incident plane wave by a sound-soft circular (resp. spherical) obstacle, for which an exact solution is known~\cite{martin2006multiple}. More precisely, letting the scattered field $u = u^\text{tot} - u^{\text{inc}}$ be given as the combined potential~\eqref{eq:combined-field-representation}, which satisfies the Sommerfeld radiation condition, we arrive at the CFIE~\eqref{eq:CFIE} with boundary data $f = -\gamma_0\uinc$ for $u^\inc(\ner) = \exp\big(\mathrm{i}\omega  \ner\cdot (1,0,0)\big)$. This setting is closer to realistic problems as it takes into account the effect of geometrical errors associated with the surface representation, which for complex objects is given by an approximation of the true geometry using e.g. triangular patches. Figure~\ref{fig:convergence-scattering-planewave} shows the relative solution error~\eqref{eq:error_ff}, the reference solution $\uref$ being this time (a truncated series approximation of) the exact solution. The iteration counts incurred by GMRES (with a relative tolerance set to $10^{{-12}}$) were found to be approximately constant, around $19 \pm 2$ for $d=2$ and $35 \pm 5$ for $d=3$, for all of the mesh sizes and values of $p$ considered. To further demonstrate that the resulting linear system preserves the well-conditioned behavior of the CFIE formulation, which yields eigenvalues $\lambda$ accumulating at~$1/2$ due to the compactness of the weakly-singular operators $S$ and $K$, we display in Figures~\ref{fig:spec-2d} and~\ref{fig:spec-3d} the spectrum of the resulting system matrix for $d=2$ and $d=3$, respectively. The clustering of the eigenvalues around $\lambda = 1/2$ is clearly visible in both cases.\enlargethispage*{1ex}

\subsection{Meshed surfaces}
\label{sec:hybrid-meshes}

The examples considered so far in this work involve simple (smooth) surfaces given as unions of non-overlapping logically-quadrilateral patches admitting analytical parametrizations. Complex three-dimensional objects of engineering interest, however, are often not available in this form. In fact, they are typically given as CAD models from which surface meshes can be produced using many mature codes. In this section, we consider the effect of using meshed surfaces produced by off-the-shelf software, demonstrating that our methodology works well with high-order (geometrical) elements generated through the freely-available mesh generation code Gmsh~\cite{geuzaine2009gmsh}.

\begin{figure}[b] \centering
  \includegraphics[width=0.85\textwidth]{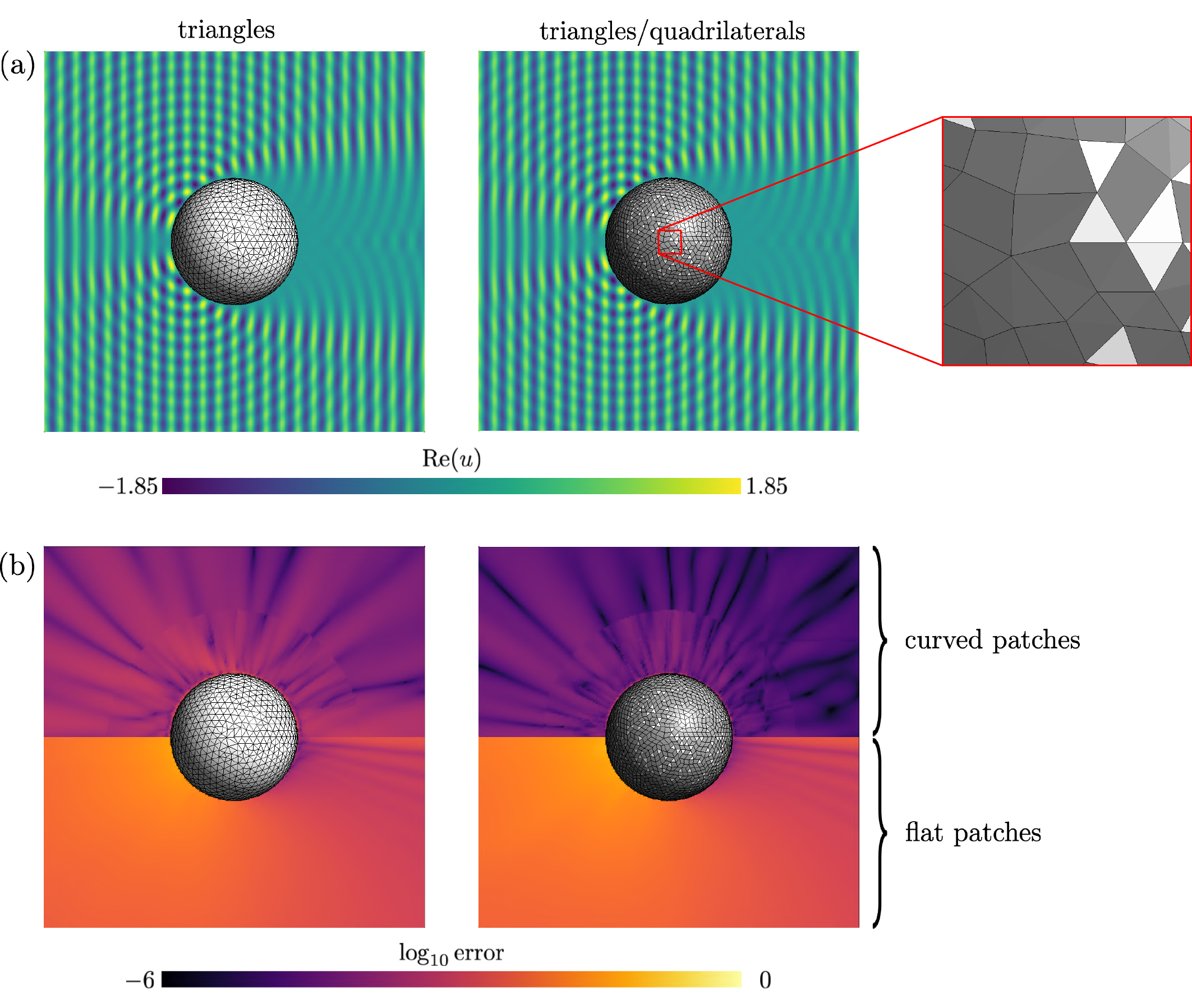}
  \caption{Solution of a scattering problem using different mesh types. The surface mesh color indicates the patch type (light gray for triangles, dark gray for quadrilaterals). The mesh either is fully triangular (left column) or consists mostly of quadrilateral patches and a few triangles (right column). The plots show the real part of the total field (top row) or the near-field errors computed using regularization (bottom row).}
    \label{fig:hybrid-meshes}
\end{figure}

To this aim, the next example explores various meshing strategies for a scattering problem onvolving a unit sphere. In particular, we consider surface meshes composed of patches of the following types:
\begin{enumerate}[noitemsep]
    \item[(a)] flat triangles,
    \item[(b)] curved (cubic) triangles,
    \item[(c)] flat hybrid triangles and quadrilateral, and
    \item[(d)] curved (cubic) hybrid triangles and quadrilaterals.
\end{enumerate}
Since our regularization technique is largely independent of the underlying geometrical representation, it is capable of handling in a unified manner surface meshes such as (d), thus providing significant flexibility. For validation purposes, we consider once again the CFIE formulation of the sound-soft scattering problem of a plane wave by the unit sphere, described in Section~\ref{sec:bvp}.

Figure~\ref{fig:hybrid-meshes}a plots the real part of the total (i.e. incident plus scattered) field obtained using patches of type (b) (left) or (d) (right, with an inset showing a zoomed view of the hybrid mesh). The 6- and 9-node quadrature rules depicted in~Figure~\ref{fig:quad_points} are used for the triangular and quadrilateral patches, respectively. As the CFIE problem~\eqref{eq:CFIE} is solved using a Nystr\"om method, the quadrature nodes are also the interpolation/collocation nodes, see Section~\ref{sec:numer-discr}. For triangular patches, these nodes define over the reference triangle a unique interpolation polynomial in $\mathbb{P}_2^{\Delta} = \mbox{span}\left\{x^{r} y^{s} : 0\leq r+s\leq 2, (r,s)\in \mathbb{N}^2 \right\}$. Similarly, for quadrilateral patches, the 9 quadrature nodes define over the reference square a unique interpolation polynomial in $\mathbb{P}_2^{\square} = \mbox{span}\left\{x^{r} y^{s} : \max\left\{r,s \right\} \leq 2, (r,s)\in \mathbb{N}^2 \right\}$. These observations suggest that the interpolation error is, in both cases, of order $\Ocal(h^3)$ where $h$ denotes the characteristic patch diameter.

Figure~\ref{fig:hybrid-meshes}b, in turn, shows the absolute errors obtained in these examples. Errors using curved (resp. flat) patches are plotted in the upper-half part (resp. lower-half part) of the panels, for triangular (left) or hybrid (right) meshes. All errors shown are evaluated against the (truncated) exact series solution. These results clearly demonstrate the improved accuracy attained by using curved patches, the number $N$ of degrees of freedom being identical. This meshing method will be used for the remaining examples of this article.

\subsection{Elastodynamic exterior Neumann problem}

We now consider a time-harmonic elastodynamic scattering problem, with physical parameters $\lambda=2$, $\mu=1$, $\rho=1$ and~$\omega=\pi$. We consider a P-wave incident field $u^\inc(\ner) = \bol{d} \exp(i k_L \bol{d} \cdot \ner)$, where $\bol{d} = (1,0,0)$ and $k_L = \omega \sqrt{\rho/(\lambda+2\mu)}$. The obstacle is a traction-free cavity, so that we impose the Neumann condition $\gamma_1 (u+u^\inc) = 0$ on $\Gamma$. We resort to an integral equation formulation based on representing the scattered displacement field $u$ by the combined field potential~\eqref{eq:combined-field-representation}. Taking the $\gamma_1$ trace of~\eqref{eq:combined-field-representation}, the following integral equation for the vector density $\varphi:\Gamma\to \C^3$ is obtained:
\begin{equation}\label{eq:elastody3D}
(i\omega)\lf\{\frac{\varphi(\bx)}{2} -K'[\varphi](\bx)\rg\}  +  T[\varphi](\bx) = - \gamma_1u^\inc(\bx),\quad\nex\in\Gamma.
\end{equation}
As in Section~\ref{sec:hybrid-meshes}, the surface $\Gamma$ is represented using curved (cubic) triangular patches over which integration/collocation is performed using 6 interior quadrature nodes.

The accurate solution of~\eqref{eq:elastody3D} requires regularization of the challenging elastodynamic hypersingular operator $T$, whose definition involves finite-part integrals. Additional difficulties arise from the fact that the hypersingular operator is not compact, leading to unfavorable spectral properties of the linear system arising from~\eqref{eq:elastody3D}. In the course of producing results for this example, we did not observe severe ill-conditioning of the linear systems, and only modestly large (always less than 100) GMRES iterations were required to meet the target tolerance. For larger problem sizes and/or higher-order quadrature rules, either analytical preconditioning (based on Calder\'on identities, see, e.g.~\cite{bruno2020regularized,chaillat:hal-02512652}) or algebraic preconditioning strategies may be required.

Figure~\ref{fig:scattering-elasticity}a displays the relative field error computed on a sphere (of radius $5$) surrounding the scatterer. The reference solution in this example was taken as a numerical approximation computed using a highly refined mesh. Two different surfaces $\Gamma$ are considered: a unit sphere, and a torus with outer radius equal to 1 and inner radius equal to 1/2. In addition, the real part of two components of the total displacement field, as well as the displacement magnitude, are displayed for $\omega=10\pi$ in Figures~\ref{fig:scattering-elasticity}b (sphere) and~\ref{fig:scattering-elasticity}c (torus). The sphere's mesh consists of $1656$ (curved) triangular patches, carrying $1656 \times 6 \times 3 = 29808$ degrees of freedom, and the GMRES solver converged within a residual tolerance of $10^{-4}$ in $53$ iterations. The torus mesh is made of $1248$ (curved) triangular patches, carrying a total of $1248\times 6 \times 3=22464$ degrees of freedom; the GMRES solver converged in $61$ iterations.\enlargethispage*{1ex}

\begin{figure}[t] \centering
  \includegraphics[width=0.75\textwidth]{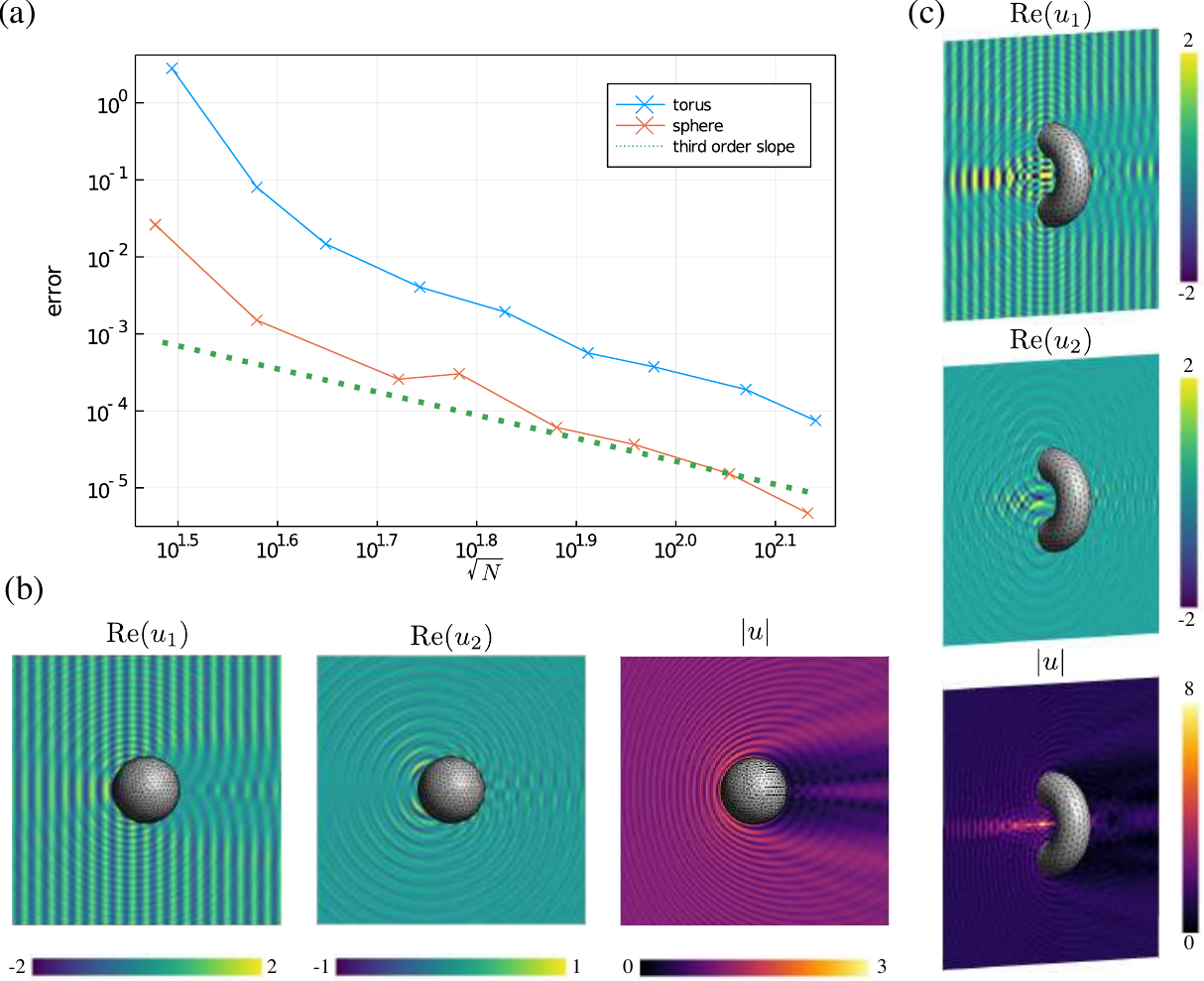}
  \caption{Time-harmonic elastodynamic scattering of a P-wave by a sphere or a torus. (a) Self convergence of the solution field evaluated on a sphere enclosing the scatterers. (b) and (c): Real part of the relevant components and magnitude of the total displacement field. In all cases the scatterer is represented using curved (cubic) triangular surface patches, and integration/interpolation within each patch is performed using a 6-point quadrature rule.}
  \label{fig:scattering-elasticity}
\end{figure}

\subsection{Large-scale exterior problem with complex geometry}

In this last set of examples we showcase the accuracy and efficiency of the overall proposed methodology used in combination with an in-house implementation of a hierarchical matrix compression algorithm~\cite{hackbusch2015hierarchical}, in conjunction with a diagonally-preconditioned GMRES~\cite{saad1986gmres}, on complex geometries of engineering interest. More precisely, we solve an exterior three-dimensional Helmholtz Neumann boundary value problem where the surface $\Gamma$ corresponds to an A319 airplane of approximate length and wingspan 34m and 36m, respectively\footnote{\url{https://gitlab.onelab.info/gmsh/gmsh/-/blob/gmsh_4_6_0/benchmarks/statreport/A319.brep}}. For the results that follow, the surface is approximated by means of (flat) triangular patches generated using Gmsh~\cite{geuzaine2009gmsh} (see Figure~\ref{fig:A319_mesh}). Throughout this section we employ a direct BIE formulation and thus solve for the unknown Dirichlet trace $\varphi =\gamma_0 u$ on the airplane surface $\Gamma$. The resulting direct second-kind BIE is given, with $g = \gamma_1 u$, by
\begin{equation}\label{eq:direct_BIE}
    -\frac{\varphi(\nex)}{2} + K[\varphi](\nex) = S[g](\nex),\quad\nex\in\Gamma.
\end{equation}

For validation purposes, we begin by constructing an exact exterior solution $u^{\rm ref}$ by placing a point source inside the airplane. The errors in the numerically produced solution $u$, which is obtained by solving~\eqref{eq:direct_BIE} for $\varphi = \gamma_0 u$ with $g=\gamma_1 u^{\rm ref}$, are assessed through~\eqref{eq:error_ff} by comparing it with $u^{\rm ref}$ at target points on a box sufficiently large to contain the airplane (see Figure~\ref{fig:A319_exact}).  The wavelength $\lambda=2\pi/\omega$ considered in this case is $\lambda=4$m, which yields about 10 wavelengths across the airplane.

In order to gauge the effect of the mesh size on the solution quality, we utilize meshes of approximate patch sizes $h=0.4$m, $0.2$m, $0.1$m, and~$0.05$m. These values give rise to triangulations made of $17,258$, $48,530$, $17,1694$ and $651,344$ patches, respectively. Additionally, the effect of the order of the method, which is directly associated with the number $P$ of quadrature nodes per element (see Section~\ref{sec:error_analysis}), is explored by considering the values $P=1,3$ and~$6$.
The large scale of the resulting linear system makes the use of fast methods essential in these cases, for which we resort here to a standard hierarchical matrix compression algorithm~\cite{hackbusch2015hierarchical}. In detail, a cluster tree is constructed so that at most~$128$ quadrature nodes are contained in each leaf box. The interaction matrices between spatially well-separated boxes in the tree are then represented in a compressed format using adaptive cross approximation (ACA) with a relative tolerance~$10^{-8}$.

\begin{figure} \centering
  \includegraphics[width=0.6\textwidth]{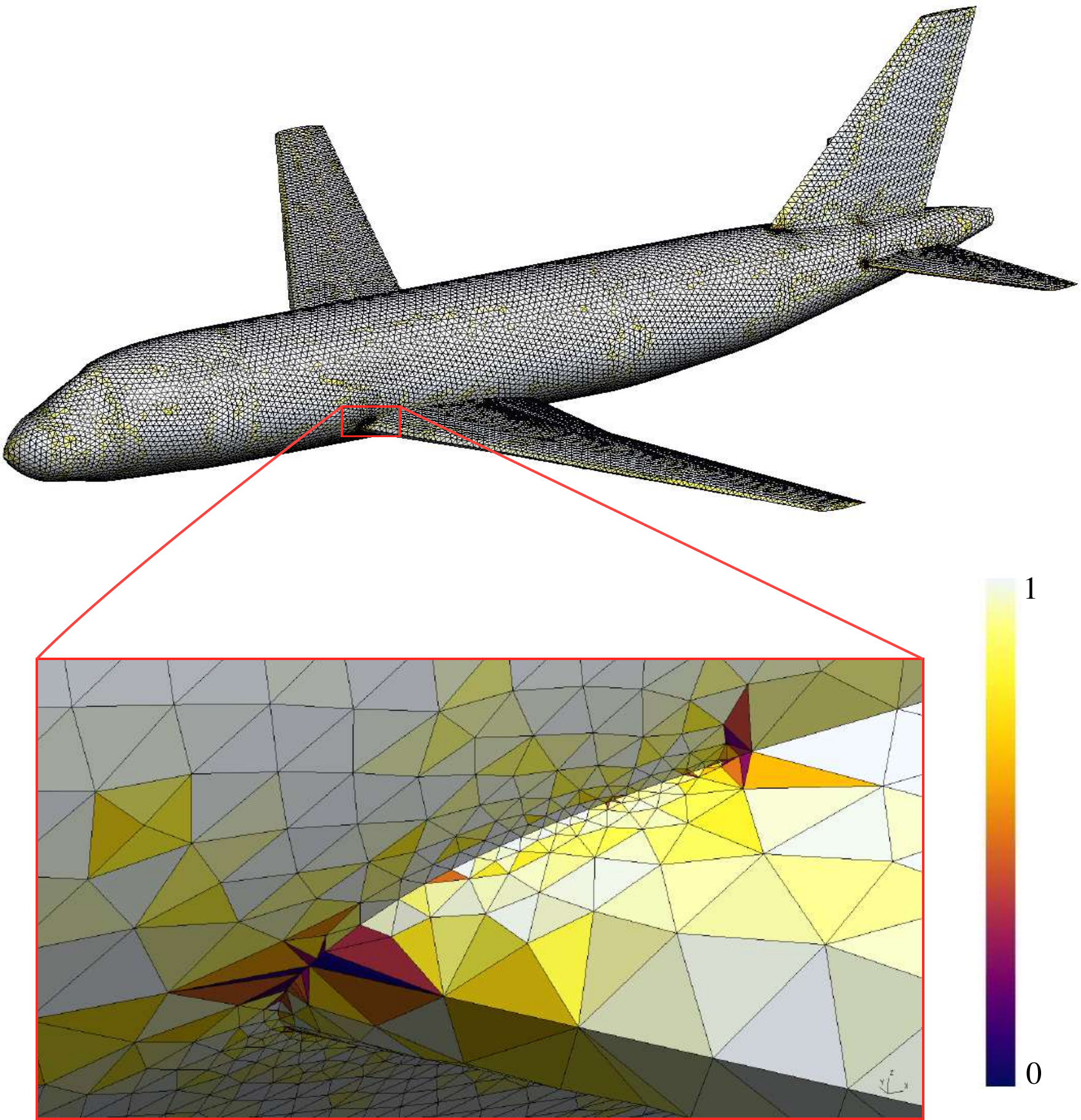}
  \caption{Airplane surface mesh with approximate mesh size $h=0.04$m utilized in some of the numerical examples reported in Table~\ref{tab:A319_error}. The elements are colored by the normalized ratio between radii of inscribed and circumscribed circles, representative of the mesh quality.}
  \label{fig:A319_mesh}
\end{figure}

Table~\ref{tab:A319_error} summarizes the results obtained, where the relative errors are displayed for  various orders and mesh sizes. Unlike the geometries so far considered, we observe a significantly larger number of GMRES iterations needed to reach the target tolerance $10^{-8}$ on the relative residual norm as $P$ increases. This is likely related to the quality of the surface meshes which contains elongated triangles as well as non-uniform patch sizes (see inset in Figure~\ref{fig:A319_mesh}). These results demonstrate the applicability of the proposed methodology to limited-quality complex meshes of engineering interest.\enlargethispage*{1ex}

\begin{remark}
The used mesh sizes correspond to about 10, 20, 40 or 80 patches per wavelength, respectively, so are (except for the coarsest one) finer than what usual engineering solution accuracy would require. Solution accuracy in BIE methods in general is strongly reliant on accurate evaluation of the singular element integrals. The high relative solution errors achieved by the refined discretizations show that our regularization methodology meets this requirement. To attain such solution accuracy levels in turn required the tighter-than-usual $10^{-8}$ tolerance on the GMRES solver, which partly accounts for the observed iteration counts.
\end{remark}

\begin{table}
  \centering
  \begin{tabular}{r|c|c|c|c}
  \toprule
    $P$ & $h$ (m) & $\#$ DOF & $\#$ iter. & error ($\%$)\\
    \hline
    1   & 0.40 & 17,258 & 270 &  4.6220  \\
    \bf{(a) 1}   & \bf{0.20} & \bf{48530} & \bf{429} &  \bf{2.1534}  \\
    1   & 0.10 & 171,694 & 478 &  0.7200 \\
    1   & 0.05 & 651,344 & 742 &  0.1730 \\
    \hline
    \bf{(b) 3}   & \bf{0.40} & \bf{51774} & \bf{526} &  \bf{0.1040} \\
    3   & 0.20 & 145,590 & 743 &  0.0339 \\
    3   & 0.10 & 515,082 & 798 &  0.0034 \\
    \hline
    6   & 0.40 & 103,548 & 932 & 0.0033    \\
    6   & 0.20 & 291,180  & 1203 & 0.0006  \\
    6   & 0.10 & 1,030,164 & 1615 & 0.0001  \\
    \bottomrule
  \end{tabular}
  \caption{Relative errors, measured using~\ref{eq:error_ff}, and numbers of GMRES iterations obtained in the solution of the point-source test problem for the A319 airplane displayed in Figure~\ref{fig:A319_mesh}, for various mesh sizes $h$ and numbers of quadrature nodes per patch $P$. The field errors corresponding to the cases marked with bold letters are displayed in Figure~\ref{fig:A319_exact}.}
  \label{tab:A319_error}
\end{table}

\begin{figure} \centering
  \includegraphics[width=1\textwidth]{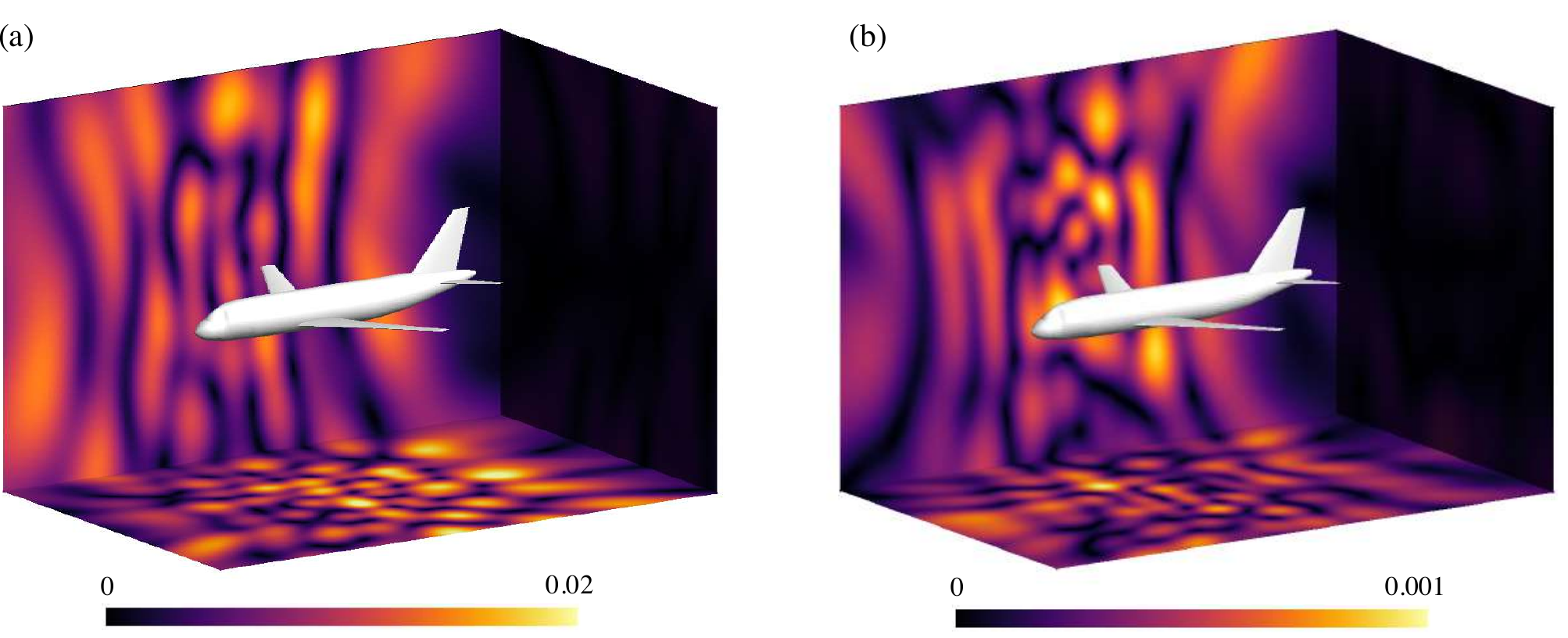}
  \caption{Field errors for $P=1$ and $h=0.2$m (left) and $P=3$, $h=0.4$m (right) corresponding to the entries marked with bold letters in Table~\ref{tab:A319_error}.}
  \label{fig:A319_exact}
\end{figure}

In our final example we solve the  Neumann (sound hard) scattering problem resulting from a plane wave $u^{\inc}(\ner) = \e^{i k \ner\cdot (1/\sqrt{2},0,-1/\sqrt{2})}$ impinging upon the airplane surface. The Dirichlet trace of the scattered field $\gamma_0 u^s$ is obtained by solving~\eqref{eq:direct_BIE} with $g=-\gamma_1 u^{\inc}$. The wavelength is taken to be $\lambda=1$m, which leads to approximately $35$ wavelengths across the airplane. A mesh size of $h=0.1$m is used with $P=3$ quadrature nodes per patch, so that the model is comprised of $N=515,082$ DOFs and features about 10 patches per wavelength. The real part and the magnitude of the total field are shown in Figures~\ref{fig:A319_scattering}(a) and \ref{fig:A319_scattering}(b), respectively. The quality of the numerical solution can be visually assessed by noticing that the normal derivative of the total field approximately vanishes at the surface of the scatterer, which due to the selected planewave direction, is observable on the top of the airplane's fuselage.

\begin{figure} \centering
  \includegraphics[width=1\textwidth]{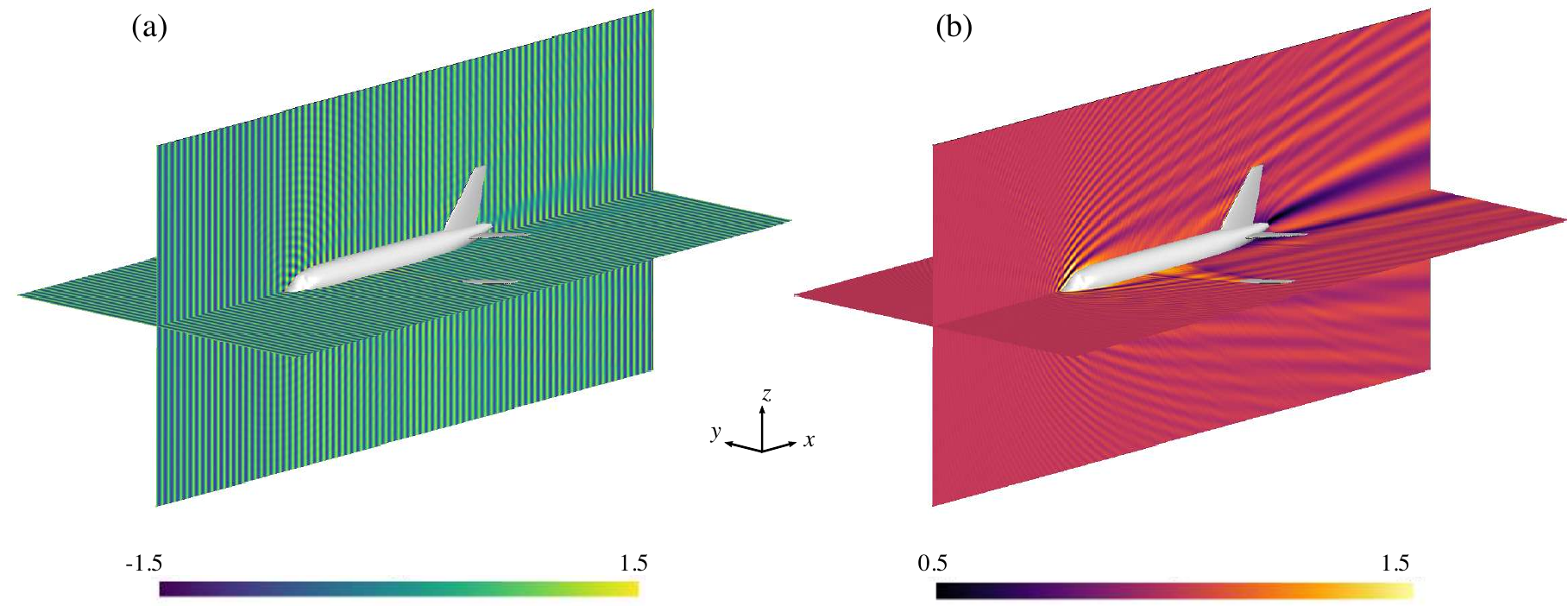}
  \caption{Sound-hard scattering off of an A319 airplane resulting from  an impinging planewave along the $(1,0,0)$ direction. The wavelength is $1$m, the mesh size is $h=0.1$m, and $P=3$ quadrature nodes per patch were used. The left figure shows the real part of the total field, while the right figure shows the magnitude of the total field.}
  \label{fig:A319_scattering}
\end{figure}

\section{Conclusions}

We introduced a general methodology that circumvents some of the technical present in previous high-order density interpolation techniques~\cite{HDI3D,perez2019planewave}. In particular, we developed a novel density interpolant that is effected by means of simple interpolation-collocation conditions that render the overall method both kernel and dimension independent. Its numerical accuracy as well as its compatibility with acceleration algorithms, was demonstrated through a series of numerical examples involving both simple surfaces of academic interest and more complex surfaces of engineering relevance. We are currently working toward extending this novel methodology to three-dimensional Maxwell equations as well as making it capable of handling challenging problems involving open surfaces (e.g., screens and cracks).

\appendix

\section{Free-space Green's functions}\label{app:free_space_Green}

For completeness we provide here the fundamental solutions $G(\br,\br') = \Gcal(\ner-\ner')$ used in this work. For time-harmonic elastodynamics, it is convenient to use the wavenumbers $k_L$ and $k_T$ of compressive (or logitudinal) and shear (or transversal) elastic waves, respectively, defined by\enlargethispage*{7ex}
\begin{equation}
  k^2_{L}=\frac{\rho\omega^2}{\lambda+2 \mu} , \qquad k_{T}=\frac{\rho\omega^2}{\mu}.
\end{equation}
The relevant two-dimensional fundamental solutions are given by:
\begin{equation}
  \Gcal(\bol R)
  = \left\{ \begin{aligned} \ \
 & -\frac{1}{2\pi}\log|\bol R|     && {\rm (Laplace)} \\
 & \frac{\mathrm{i}}{4}H_0^{(1)}(\omega|\bol R|) && {\rm (Helmholtz)} \\
 & \frac{\lambda+3 \mu}{4 \pi \mu(\lambda+2 \mu)}\left\{-\log|\bol R|\bold{I}+\frac{\lambda+\mu}{\lambda+3 \mu} \frac{1}{|\bol R|^{2}}\bol R\bol R^{\top}\right\} && {\rm (elastostatics)} \\
 & \frac{\mathrm{i}}{4\mu}\left\{A \bold{I} + B\bol R\bol R^{\top}\right\} && {\rm (elastodynamics)}
    \end{aligned} \right.
\end{equation}
with the auxiliary functions $A$ and $B$ given by
\begin{align}
  A & =H_{0}^{1}\left(k_{T} |\bol R|\right)-\left[H_{1}^{1}\left(k_{T} |\bol R|\right)-\frac{k_L^2}{k_T^2} H_{1}^{1}\left(k_{L} |\bol R|\right)\right], \\
  B & =-2 A-\left[H_{0}^{1}\left(k_{T} |\bol R|\right)-\frac{k_L^2}{k_T^2} H_{0}^{1}\left(k_{L} |\bol R|\right)\right].
\end{align}
Similarly, the three-dimensional fundamental solutions are given by:
\begin{equation}
  \Gcal(\bol R)
  = \left\{ \begin{aligned} \ \
 & \frac{1}{4\pi|\bol R|} && {\rm (Laplace)}, \\
 & \frac{\e^{i\omega |\bol R|}}{4\pi  |\bol R|} && {\rm (Helmholtz)} \\
 & \frac{\lambda+3 \mu}{8 \pi \mu(\lambda+2 \mu)}\left\{\frac{1}{|\bol R|}\bold{I}+\frac{\lambda+\mu}{\lambda+3 \mu} \frac{1}{|\bol R|^{3}}\bol R\bol R^{\top}\right\} && {\rm (elastostatics)} \\
 & \frac{1}{4 \pi\mu |\bol R|}\left\{A \bold{I} + B\bol R\bol R^{\top}\right\} && {\rm (elastodynamics)} \\
            \end{aligned} \right.
\end{equation}
with the auxiliary functions $A$ and $B$ now given by
\begin{align}
  A & =\left(1+\frac{\mathrm{i}}{k_T|\bol R|}-\frac{1}{k^2_T|\bol R|^{2}}\right) \mathrm{e}^{\mathrm{i} k_T|\bol R|}-\frac{k_L^2}{k_T^2}\left(\frac{\mathrm{i}}{k_L|\bol R|}-\frac{1}{k_L^2|\bol R|^{2}}\right) \mathrm{e}^{\mathrm{i} k_L|\bol R|}       \\
  B & =\left(\frac{3}{k_T^2|\bol R|^{2}}-\frac{3 \mathrm{i}}{k_T|\bol R|}-1\right) \mathrm{e}^{\mathrm{i} k_T|\bol R|}-\frac{k_L^2}{k_T^2}\left(\frac{3}{k^2_L|\bol R|^{2}}-\frac{3 \mathrm{i}}{k_L|\bol R|}-1\right) \mathrm{e}^{\mathrm{i} k_L|\bol R|}
\end{align}

\bibliographystyle{abbrv}
\bibliography{References,marc}

\end{document}